\documentclass{amsart}
   \usepackage{amsmath}
   \usepackage{amsfonts}   
   \usepackage{amssymb}    
   \usepackage{amsthm}
   \usepackage{graphicx}
\begin{document}

\title{Unusual Geodesics in generalizations of Thompson's Group
$F$}
\author{Claire Wladis}
\address{Department of Mathematics, BMCC/CUNY, Department of
Mathematics, 199 Chambers St., New York, NY 10007}
\email{cwladis@gmail.com} 
\thanks{The author would like to thank Sean Cleary for his
support and advice during the preparation of this article and
the anonymous reviewer for their helpful suggestions during the
revision process. The author also acknowledges support from the
CUNY Scholar Incentive Award. }

\makeatletter
\def\maxwidth{%
  \ifdim\Gin@nat@width>\linewidth
    \linewidth
  \else
    \Gin@nat@width
  \fi
} \makeatother

\newcommand {\lft} {\mbox{${\mathcal L}$}}
\newcommand {\rt} {\mbox{\ensuremath{{\mathcal R}}}}
\newcommand {\m} {\mbox{\ensuremath{{\mathcal M}}}}
\newcommand {\mi} {\mbox{${\mathcal M^{\it i}}$}}
\newcommand {\mj} {\mbox{${\mathcal M^{\it j}}$}}
\newcommand {\mpp} {\mbox{${\mathcal M^{\it p}}$}}
\newcommand {\mone} {\mbox{${\mathcal M}^1$}}
\newcommand {\mtwo} {\mbox{${\mathcal M}^2$}}
\newcommand {\mpminusone} {\mbox{${\mathcal M}^{\it p-1}$}}

\newcommand {\lftnot} {\mbox{${\mathcal L_{\emptyset}}$}}
\newcommand {\lftl} {\mbox{${\mathcal L_{\rm L}}$}}
\newcommand {\rtnot} {\mbox{${\mathcal R_{\emptyset}}$}}
\newcommand {\rtr} {\mbox{${\mathcal R_{\rm R}}$}}
\newcommand {\rtj} {\mbox{${\mathcal R_{\it j}}$}}
\newcommand {\rtjone} {\mbox{${\mathcal R_{\it j_1}}$}}
\newcommand {\rtjtwo} {\mbox{${\mathcal R_{\it j_2}}$}}
\newcommand {\minot} {\mbox{${\mathcal M^{\it
i}_{\emptyset}}$}}
\newcommand {\mionenot} {\mbox{${\mathcal M^{\it
i_1}_{\emptyset}}$}}
\newcommand {\mitwonot} {\mbox{${\mathcal M^{\it
i_2}_{\emptyset}}$}}
\newcommand {\mij} {\mbox{${\mathcal M^{\it i}_{\it j}}$}}
\newcommand {\mione} {\mbox{${\mathcal M}^{\it i}_1$}}
\newcommand {\mionejone} {\mbox{${\mathcal M^{\it i_1}_{\it
j_1}}$}}
\newcommand {\mitwojtwo} {\mbox{${\mathcal M^{\it i_2}_{\it
j_2}}$}}
\newcommand {\mionek} {\mbox{${\mathcal M^{\it i_1}_{\it
k}}$}}
\newcommand {\mitwok} {\mbox{${\mathcal M^{\it i_2}_{\it
k}}$}}
\newcommand {\rti} {\mbox{${\mathcal R_{\it i}}$}}
\newcommand {\rtk} {\mbox{${\mathcal R_{\it k}}$}}
\newcommand {\rtone} {\mbox{${\mathcal R}_1$}}
\newcommand {\riplusone} {\mbox{${\mathcal R}_{i+1}$}}
\newcommand {\rtp} {\mbox{${\mathcal R_{\it p}}$}}
\newcommand {\rtstar} {\mbox{${\mathcal R_*}$}}
\newcommand {\mjnot} {\mbox{${\mathcal M^{\it
j}_{\emptyset}}$}}
\newcommand {\mknot} {\mbox{${\mathcal M^{\it
k}_{\emptyset}}$}}
\newcommand {\mpnot} {\mbox{${\mathcal M^{\it
p}_{\emptyset}}$}}
\newcommand {\mik} {\mbox{${\mathcal M^{\it i}_{\it k}}$}}
\newcommand {\mkl} {\mbox{${\mathcal M^{\it k}_{\it l}}$}}
\newcommand {\mlk} {\mbox{${\mathcal M^{\it l}_{\it k}}$}}
\newcommand {\mpq} {\mbox{${\mathcal M^{\it p}_{\it q}}$}}
\newcommand {\mpi} {\mbox{${\mathcal M^{\it p}_{\it i}}$}}
\newcommand {\mpj} {\mbox{${\mathcal M^{\it p}_{\it j}}$}}
\newcommand {\mpk} {\mbox{${\mathcal M^{\it p}_{\it k}}$}}
\newcommand {\mistar} {\mbox{${\mathcal M^{\it i}_*}$}}
\newcommand {\mpstar} {\mbox{${\mathcal M^{\it p}_*}$}}
\newcommand {\mlnot} {\mbox{${\mathcal M^{\it
l}_{\emptyset}}$}}
\newcommand {\mrs} {\mbox{${\mathcal M^{\it r}_{\it s}}$}}

\newcommand {\gen} {\mbox{$\{x_0,\dots,x_p\}$}}
\newcommand {\genpm} {\mbox{$\{x_0^\pm1,\dots,x_p^\pm1\}$}}

\newcommand{\x}[1]{\mbox{$x_{#1}$}}
\newcommand{\xinv}[1]{\mbox{$x_{#1}^{-1}$}}
\newcommand{\xpm}[1]{\mbox{$x_{#1}^{\pm 1}$}}

\newtheorem {lengththm}{Theorem}[section]
\newtheorem{onecaret}[lengththm]{Theorem}
\newtheorem{addcaret}[lengththm]{Theorem}

\newtheorem{noaddedcarets}[lengththm]{Remark}

\newtheorem{seesawexist}{Corollary}[section]
\newtheorem{kfellowtrav}[seesawexist]{Proposition}
\newtheorem{notcombable}[seesawexist]{Theorem}
\newtheorem{likef2}[seesawexist]{Remark}
\newtheorem{notreg}[seesawexist]{Theorem}
\newtheorem{conetype}[seesawexist]{Definition (cone type)}
\newtheorem{infconetype}[seesawexist]{Theorem}
\newtheorem{finmanyconetypes}[seesawexist]{Lemma}

\newtheorem{formdeadends}{Theorem}[section]
\newtheorem{deadends}[formdeadends]{Definition (dead ends)}
\newtheorem{depth}[formdeadends]{Definition (depth of a dead
end element)}
\newtheorem{deadenddiag}[formdeadends]{Theorem}
\newtheorem{depthdeadends}[formdeadends]{Theorem}

\newtheorem{notMAC}{Theorem}[section]
\newtheorem{sufficientdistance}[notMAC]{Lemma}
\newtheorem{hrexists}[notMAC]{Lemma}
\newtheorem{lengthleftsided}[notMAC]{Lemma}
\newtheorem{criticalline}[notMAC]{Definition}
\newtheorem{hlonpath}[notMAC]{Lemma}
\newtheorem{last3caretsofhll}[notMAC]{Lemma}
\newtheorem{lengthofleftsided}[notMAC]{Lemma}
\newtheorem{hrisleftsided}[notMAC]{Remark}
\newtheorem{defny}[notMAC]{Definition}
\newtheorem{leftsidedcommutewy}[notMAC]{Remark}
\newtheorem{lengthofyl}[notMAC]{Lemma}
\newtheorem{hl=yhl}[notMAC]{Lemma}
\newtheorem{hrinvhlprime}[notMAC]{Lemma}
\newtheorem{lengthhrinvhlprime}[notMAC]{Corollary}
\newtheorem{MAC}[notMAC]{Theorem}
\newtheorem{almostconvex}[notMAC]{Definition (almost convex)}
\newtheorem{minalmostconvex}[notMAC]{Definition (minimally
almost convex)}
\newtheorem{rightfoot}[notMAC]{Definition (right foot)}
\newtheorem{critleaf}[notMAC]{Definition (critical leaf)}
\newtheorem{leftsided}[notMAC]{Definition (left-sided)}
\newtheorem{width}[notMAC]{Definition (width of a group
element)}

\newtheorem{thm}{Theorem}[section]
\newtheorem{lem}{Lemma}[section]
\newtheorem{cor}{Corollary}[section]
\newtheorem{rmk}{Remark}[section]
\newtheorem{defn}{Definition}[section]
\newtheorem{nota}{Notation}[section]

\rm

\begin{abstract}
We prove that seesaw words exist in Thompson's Group $F(N)$ for
$N=2,3,4,...$ with respect to the standard finite generating
set $X$.  A seesaw word $w$ with swing $k$ has only geodesic
representatives ending in $g^k$ or $g^{-k}$ (for given $g\in
X$) and at least one geodesic representative of each type.  The
existence of seesaw words with arbitrarily large swing
guarantees that $F(N)$ is neither synchronously combable nor
has a regular language of geodesics.  Additionally, we prove
that dead ends (or $k$--pockets) exist in $F(N)$ with respect
to $X$ and all have depth 2.  A dead end $w$ is a word for
which no geodesic path in the Cayley graph $\Gamma$ which
passes through $w$ can continue past $w$, and the depth of $w$
is the minimal $m\in\mathbb{N}$ such that a path of length
$m+1$ exists beginning at $w$ and leaving $B_{|w|}$. We
represent elements of $F(N)$ by tree-pair diagrams so that we
can use Fordham's metric.  This paper generalizes results by
Cleary and Taback, who proved the case $N=2$.
\end{abstract}

\maketitle \tableofcontents

\section{Generalizations of Thompson's groups $F$}
\subsection{Introduction}
\par Thompson's group $F(N)$ is a generalization of the group
$F$, which
R. Thompson introduced in the early 1960's (see \cite{F}) while
constructing the groups $V$ and $T$ (also often referred to in
the literature as Thompson's groups), which were the first
known examples of infinite, simple, finitely-presented groups.
Here $F\subseteq T\subseteq V$.  Higman in \cite{Higman} later
generalized $T$ into an infinite class of groups, and Brown
applied this same generalization to the groups $F$ and $V$ in
\cite{Brown}.  This paper only considers generalizations of the
group $F$.

\begin{defn}[Thompson's group $F(N)$]
Thompson's group $F(N)$, for $N\in\{2,3,4,...\}$, is the group
of piecewise-linear orientation-preserving homeomorphisms of
the closed unit interval with finitely-many breakpoints in the
ring $\mathbb{Z}[\frac{1}{N}]$ and slopes in the cyclic
multiplicative group $\langle N\rangle$ in each linear piece.
\end{defn}

\par $F$ is then simply the group $F(2)$.  Throughout this
paper, we use the convention that $N=p+1$ for
$p\in\mathbb{Z}_+$ (we note that $p$ need not be prime, but is
rather a positive integer); this is because the numbering of
tree-pair diagrams and some algebraic expressions will be
simpler with the use of $p$ rather than $N$.

\par $F(p+1)$, $p\in\mathbb{N}$, is
finitely-presented, infinite-dimensional, torsion-free and of
type $FP_\infty$ (see \cite{Brown2}).  This paper is
specifically interested in the Cayley graph of $F(p+1)$ with respect to the standard
finite generating set, about which relatively little is known.
One known result is that $F(p+1)$  satisfies no nontrivial
convexity condition with respect to the standard finite
generating set (see \cite{F2notMAC}, \cite{FnotAC}, and
\cite{notMAC}).  More detailed information about Thompson's
groups can
be found in \cite{intronotes}.

\subsection{Unusual geodesics}
\par The first unusual kind of geodesic behavior in $F(p+1)$ to
be explored in this paper is illustrated by the existence of
seesaw words.

\subsubsection{Seesaw words}
\par Groups with seesaw words with arbitrarily large swing
are not synchronously combable by geodesics and do not have a
regular language of geodesics.  In \cite{seesaw}, Cleary and
Taback show that Thompson's group $F(2)$ has seesaw words of
arbitrarily large swing; we generalize this
argument to $F(p+1)$ for $p>2$.  Cleary and Taback
have also shown in \cite{lamplight} that the Lamplighter groups
and certain generalized wreath products also have
seesaw words of arbitrarily large swing.

\begin{defn}[seesaw word]\label{seesawword}
A word $w$ with length $|w|$ is a seesaw word with swing
$k\in\mathbb{N}$ with respect to $g$ in generating set $X$
if the following hold:
\begin {enumerate}
\item $|wg^l|=|w|-|l|$ for $0<|l|\le k$ \item
$|wg^lh|\ge|wg^l|$ for all $h\in X\cup X^{-1}$ such that
$h\ne g$, when $0<|l|<k$
\end{enumerate}
In other words, all geodesic representatives of a seesaw word
$w$ end in either $g^k$ or $g^{-k}$, and there is at least one
geodesic representative of each type.
\end{defn}

\begin{defn}[(synchronous) k-fellow traveller
property)]\label{kfeltravdef}
Let $\lambda$ and $\eta$ be geodesic paths in the Cayley graph
$\Gamma(G,X)$ that the identity to $w$
and $v$, respectively.  Then $\lambda$ and $\eta$
(synchronously) k-fellow travel if for some constant $k$:
\begin{enumerate}
\item $d_{\Gamma}(w,v)=1$ and \item For any 2 vertices $h$ on $\lambda$
and $g$ on $\eta$, if
$|h|=|g|$, then $d_{\Gamma}(h,g)\le k$.
\end{enumerate}
\end {defn}

\begin{defn}[(synchronously) combable]\label{combable}
\par A group is {\it (syn.) combable} if
it can
be represented by a language of words satisfying the
(syn.) k-fellow traveller property.
\end{defn}

\subsubsection{Dead ends}
\par Dead ends were first defined by Bogopolski in 1997 in
\cite{bogo}.  Any geodesic representative of a dead end
word cannot be extended past that word in the Cayley graph.
The depth of a dead end then measures how severe this behavior
is: for a dead end element $w$ of length $m$, a depth of $k$
means that only paths beginning at $w$ of length greater than
$k$ can leave the ball $B_m$.

\begin{defn}[dead ends] \label{deadends}
An element $w$ of a group $G$ is a dead end with respect to the
given generating set $X$ if $|wg^{\pm 1}|\le|w|$ for all $g\in
X$.
\end{defn}

\par In this
paper we give a general form for all dead end elements in
$F(p+1)$.

\begin{defn}[depth of a dead end element]\label{depth}
For a dead end element $w$, let $|w|=n$.  The depth of a dead
end element $w$ in the generating set $X$ is the smallest
number $m$ such that $|wg_1\cdots g_{m+1}|\le n$ for all
possible $g_1,...,g_{m+1}\in X\cup X^{-1}$.  If no such $m$
exists, we say that the dead end has infinite depth.

In other words, the depth of a dead end is the smallest integer
$m$ such that all paths of length $m$ or less emanating from
$w$ remain in the ball $B_n$ (centered at the identity), but
for which there exists a path of length $m+1$ which leaves
$B_n$.
\end{defn}

\par Clearly all dead ends have depth greater than or equal to
1 (and for groups with all relators of even length this depth is
greater than or equal to 2).  If a group has a dead end $w$
with depth $k\ge1$, we
can also say that $w$ is a {\it k--pocket} in the Cayley graph
of the group. We will show that while $F(p+1)$ has dead ends,
it does not have deep k--pockets, because all dead ends in
$F(p+1)$ have depth 2.

\par The property of having dead ends has been explored for
several groups already.  Thompson's group $F(2)$ has dead ends,
all of which have depth 2, as Cleary and Taback show in
\cite{combpropF}; our results simplify to this case when $p=1$.
In contrast, dead ends with arbitrary depth exist in the
Lamplighter groups, and in some more general wreath products
with respect to the natural generating sets (see \cite
{lamplight}).

\subsection{Tree-pair diagram representatives}
\par What follows for the remainder of this section is
summarized from \cite{notMAC}; greater detail can be found there.

\par Because elements of $F(p+1)$ are piecewise linear maps which
take the $i$th subinterval of the domain to the $i$th
subinterval of the range, any
element of $F(p+1)$ is wholly determined by
the subdivisions present in its domain and range.  In fact, any
element $x\in F(p+1)$ can be entirely determined by an ordered
pair of two sets of consecutive subintervals of $[0,1]$:
$$(D=\{I_0=[a_0,a_1],...,I_k=[a_k,a_{k+1}]\},R=\{J_0=[b_0,b_1],...,J_k=[b_k,b_{k+1}]\})$$
where $a_i<a_{i+1},b_i<b_{i+1}$ for all $i\in\{0,...,k+1\}$,
and $x$ is the map that takes $I_i$ to $J_i$ for all
$i=1,...,k$.  {/it Tree-pair diagrams}, which we will use to
represent elements of $F(p+1)$, are a geometric
representation of this idea.

\par A graph of $p+2$ vertices, one with degree
$p+1$ ({\it parent vertex}) and the rest with degree 1 ({\it child vertices}), and $p+1$
directed edges is a {\it $(p+1)$--ary caret}.  A diagram which consists of $(p+1)$--ary carets, each with parent vertex oriented upwards and sharing at least one vertex with another caret, is called a {\it $(p+1)$--ary
tree}.  The graph consisting of an ordered pair of $(p+1)$--ary
trees with the same number of leaves (or equivalently the same number of carets) is a {\it $(p+1)$--ary
tree-pair diagram}.

\begin{defn}[nodes and leaves]\label{nodesleaves} Within a
$(p+1)$--ary tree, any vertex which is the parent
vertex of a caret (i.e. which has degree $p+1$ or $p+2$) is a {\it
node}; any
vertex which has degree 1 is a {\it leaf}.  We note that here, the term node refers only to vertices which are not
leaves; it is not a synonym for vertex.
\end{defn}

\par The top
node of a $(p+1)$--ary tree is the {\it root} or {\it root
node}, and the caret which contains it is called the {\it root
caret}.  We refer to the
leftmost or rightmost directed edge of a tree as the {\it left} or {\it
right edge} of the tree respectively.

\subsubsection{Leaf ordering in a tree-pair diagram}
\par We recall that an arbitrary element $x$ of $F(p+1)$ can be
entirely determined by an ordered pair of sets of consecutive
subintervals of $[0,1]$:
$(D=\{I_0,...,I_k\},R=\{J_0,...,J_k\})$.  Each leaf in a
tree-pair diagram will correspond to one of the intervals
$I_0,...,I_k,J_0,...,J_k$ in the following way: if the parent
node of a caret represents an interval $[a,b]$, then the child
nodes of that caret represent the subintervals
$[a,a+\frac{b-a}{p+1}],...,[a+\frac{p(b-a)}{p+1},b]$; we let
the root node of each tree in a tree-pair diagram represent
$[0,1]$, so each leaf in the first (or second) tree in the the
tree-pair diagram now represents a subinterval $I_0,...,I_k$
(or $J_0,...,J_k$).  We then number the leaves in the tree by
assigning each of them the index number of the interval which
they represent.  For more details see
\cite{notMAC}.  We can see a tree-pair diagram with all
its leaves numbered in Figure \ref{exampletree}.

\begin{figure}[htbp]
\centering
\includegraphics[width=4.5in]{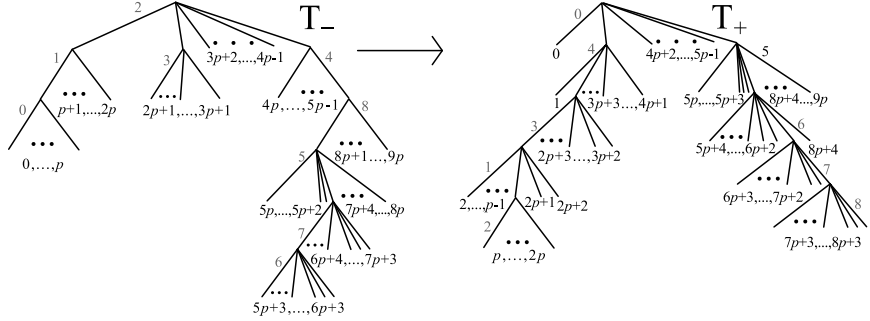}
\caption{Tree-pair diagram representative of an element of
$F(p+1)$ with all carets and leaves numbered.}
\label{exampletree}
\end{figure}

\subsubsection{Minimal tree-pair diagrams}
\par The group $F(p+1)$ induces an equivalence relation on the
set of $(p+1)$--ary tree-pair diagrams.

\begin{defn}[equivalent tree-pair diagrams]\label{equivtpds}
Two $(p+1)$--ary tree-pair diagrams are {\it equivalent} if
they represent the same element of $F(p+1)$.
\end{defn}

\begin{defn}[minimal tree-pair diagram
representative]\label{mintpdrep}
The tree-pair diagram which has the smallest number of leaves
of any diagram in its equivalence class is the {\it minimal
tree-pair diagram representative} of the element of $F(p+1)$
represented by that equivalence class.
\end{defn}

\par Within a $(p+1)$--ary tree-pair diagram, the domain tree
is
referred to as the {\it negative} tree and is often denoted by
$T_-$, whereas the range tree is referred to as the {\it
positive} tree and is denoted by $T_+$.  We will denote a
tree-pair diagram with negative tree $T_-$ and positive tree
$T_+$, by $(T_-,T_+)$.

\par We describe how we may obtain the equivalent minimal
tree-pair diagram representative of an element of $F(p+1)$ from
an arbitrary representative.
We say that a caret is {\it exposed}
if all of its children are leaves.  If there is an exposed
caret in both the negative and positive trees, and all the
leaves of the exposed caret in each tree have the same index
numbers, then we can remove the pair of exposed carets in the
tree-pair diagram because this does not change the element
which the tree-pair diagram represents.  This is the only way
in which a tree-pair diagram can be reduced. So, every element
of $F(p+1)$ has a unique representation as a minimal tree-pair
diagram.  We will write $w=(T_-,T+)$ to denote that $(T_-,T_+)$
is the minimal tree-pair diagram representative of $w$.

\begin{nota}[$((Tx)_-,(Tx)_+)$,
$((Tx)'_-,(Tx)_+)'$]\label{prodtpd}
When $w=(T_-,T_+)$ and $x\in F(p+1)$, we denote the (possibly
non-minimal) tree-pair diagram representative of the product
$wx$ by $((Tx)_-,(Tx)_+)$.  We will denote the minimal
tree-pair diagram representative of $wx$ by
$((Tx)'_-,(Tx)'_+)$.
\end{nota}

\subsubsection{Multiplying tree-pair diagrams}
\par Multiplication of two elements of $F(p+1)$ is simply
function
composition.  We will use functional notation so that
multiplying $x$ by $y$ on the right will be written $xy$, which
denotes $x\circ y$.

\par To compute the product $xy$ of $x=(T_-,T_+)$ and
$y=(S_-,S_+)$ using the tree-pair diagram
representatives, we first make $S_+$ identical to
$T_-$. This is possible because we can add a caret to any leaf
in $S_+$ as long as we add a caret to the leaf with the same
index number in $S_-$, because this is just the reverse of the
process removing exposed caret pairs. In the same way, we can
add a caret to any leaf in $T_-$.  We continue adding carets to
the tree-pair diagrams in this way until $T_-$ and $S_+$ are
identical. If we let $(T_-^*,T_+^*)$ and $(S_-^*,S_+^*)$ denote
the tree-pair diagrams for $x$ and $y$ respectively once carets
have been added as needed so that $S_+^*=T_-^*$, then
$(S_-^*,T_+^*)$ is the (possibly non-minimal) tree-pair diagram
representative of $xy$.  To see an example of multiplication of tree-pair diagrams, see Figure \ref{mult}.

\begin{figure}[t]
\centering
\includegraphics[width=3in]{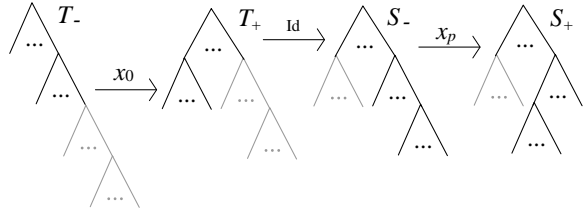}
\caption{Multiplication of tree-pair diagrams representing the
product $x_px_0$ in $F(p+1)$ (each caret has $p+1$ edges) where $x_0=(T_-,T_+),x_p=(S_-,S _+)$.
Here $T_-,T_+,S_-,S_+$ are the trees represented by only black
carets.  $T_-^*,T_+^*,S_-^*,S_+^*$ are then the trees
represented by the union of black and grey carets, and
$x_px_0=(T_-^*,S_+^*)$.  }\label{mult}
\end{figure}

\subsection{Caret types}\label{types}
\par In order to understand the metric on $F(p+1)$ developed by
Fordham in \cite{length}, which we will need to prove the
results of this paper, we must first categorize the carets in a
tree into the following
types:
\begin{enumerate}
\item \lft.  This is a left caret; a left caret is any
caret that has one edge on the left side of the tree.
The root caret is defined to be of this type. \item
\rt.  This is a right caret; a right caret is any caret
(except the root caret) that has one edge on the right
side of the tree. \item \m.  This is a middle caret;
all carets which are neither left nor right carets are
middle carets.
\end{enumerate}

\subsection{Group presentations}
\par $F(p+1)$ has a standard infinite presentation and
a standard finite presentation; the infinite presentation can
be obtained from the finite presentation by induction.

\par The standard infinite
presentation is \cite{Brown}:
$$F(p+1)=\{x_0,x_1,x_2,...|x_ix_j=x_{j+p}x_i\hbox{ for
}i<j\}$$

\begin{figure}[htbp]
\centering
\includegraphics[width=3in]{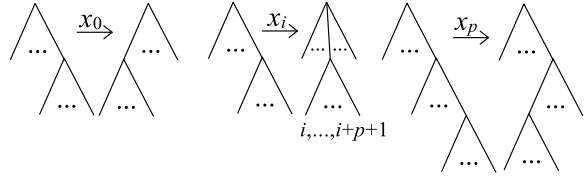}
\caption{The standard finite generators of $F(p+1)$, where
$i\in\{1,...,p-1\}$ (each caret has $p+1$ edges).} \label{geninf}
\end{figure}

\par The standard
finite presentation is \cite{Brown} (see Figure \ref{geninf}):

$$\{x_0,x_1,...,x_{p}\left|\begin{array}{l}
    [x_0x_i^{-1},x_j]\hbox{ when } i<j,
[x_0^2x_i^{-1}x_0^{-1},x_j]\hbox{ when } i\ge j-1,\\
    {[x_0^3x_{p}^{-1}x_0^{-2}, x_1]}
    \hbox{. Here }i,j=0,...,p.
    \end{array}
    \right\}
$$

\par From now on we will use the notation
$X$ to represent the generating set \gen.

\par In \cite{length}, Fordham developed a metric to calculate
geodesic lengths in the Cayley graph of $F(p+1)$ generated by
$X$ (this is a generalization of
his work in \cite{fordhamthesis} and \cite{fordhamgd}).  The
material in this section is primarily paraphrased from
\cite{length}.  This metric depends upon the exact types of
carets within a $(p+1)$--ary tree, so before we proceed to
present the metric, we further classify caret types.

\subsection{Further Classification of Carets of type
\m}\label{mtypes}
\par We further subcategorize the middle carets into $p$
subtypes: \mi\ for $i=1,2,...,p$.  The value of $i$ depends
upon the type of the middle caret's parent caret and its
relative location with respect to its parent caret.  Figure
\ref{type} shows the subtype of each child caret for a given
parent caret type.  For example, in Figure \ref{exampletree},
$\wedge_3, \wedge_5, \wedge_6, \wedge_7\in T_-$ have types
$\m^1, \m^p, \m^3, \m^3$ respectively, and
$\wedge_1,\wedge_2,\wedge_3,\wedge_4,\wedge_6,\wedge_7,\wedge_8\in
T_+$ have types $\m^2,\m^p,\m^2,\m^1,\m^4,\m^3,\m^3$
respectively.

\begin{figure}[t]
\centering
\includegraphics[width=3.5in]{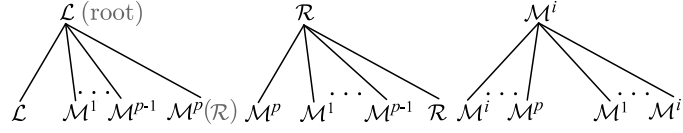}
\caption{For each of the parent caret types given above: \lft,
\rt, and \mi\ for $i=2,...,p$, the caret type listed below each child is the
type of the child caret in that position, if one exists.}
\label{type}
\end{figure}

\subsection{Caret/Node order}
\par The metric is based on numbering all the carets in each
tree of a tree-pair diagram and pairing up each caret in the
negative tree with the caret in the positive tree with the same
index number.  The type of each caret in the pair then
determines the contribution of that pair of carets to the
length of the element which the tree-pair diagram represents.

\begin{defn}[ancestor, descendant]
For any two vertices $a$ and $b$ on an $n$-ary tree, vertex $a$
is the {\it ancestor} of vertex $b$ if it is on the directed
path from the root node to vertex $b$. Similarly, vertex $b$ is
the {\it descendent} of vertex $a$ if vertex $a$ is the
ancestor of vertex $b$.
\end{defn}

\par To order the carets in a $(p+1)$--ary tree, we first order
the nodes of the tree.  Once we have ordered the nodes within a
tree, we can simply number them, beginning with 0 and assigning
numbers so that the numbering reflects the placement of the nodes in the order.  And once we have numbered the nodes of a
tree, we can number the carets in the tree simply by assigning
to each caret the index number of its parent node.

\par To order all the nodes within a tree, we begin by ordering
all the nodes within a single caret.  Since every caret in a
tree has at least one node which is common to another caret in
the tree, any absolute order for the nodes within an arbitrary
caret induces an absolute order on all the nodes in a tree
(i.e. for any 3 nodes within a single caret $a,b,c$ such that
$a<b<c$ in the order, for an arbitrary descendant node $b'$ of
$b$, we must also have $a<b'<c$).

\par Now we describe this absolute order of nodes within a
caret.  The type of a given caret
determines which child nodes will come before the parent node
in the order and which will come after it (see Figure
\ref{type2}). For an arbitrary caret, we assign index numbers
$\alpha _0,...,\alpha_{p+1}$ to every vertex within the caret;
how these index numbers will be assigned depends upon the caret
type: For left and right carets, the leftmost child vertex of
the caret will have index number $\alpha_0$, the root vertex
will have index number $\alpha_1$, and the remaining child
vertices will have index numbers $\alpha_2,...,\alpha_{p+1}$.  For carets of type \mi,
the $p-i+1$ leftmost child vertexes will have index numbers
$\alpha_0,...,\alpha_{p-i}$, the parent vertex will have index number
$\alpha_{p-i+1}$, and the remaining child vertices will have
index numbers $\alpha_{p-i+2},...,\alpha_{p+1}$.  For a visual summary of these
details, see Figure \ref{type2}.  Then these vertex index
numbers induce an ordering of the nodes of the caret as
follows: for arbitrary nodes $a$ and $b$ in the caret with
vertex index numbers $\alpha_j$ and $\alpha_k$, $a<b$ if and
only if $j<k$.

\begin{figure}[t]
\centering
\includegraphics[width=2.5in]{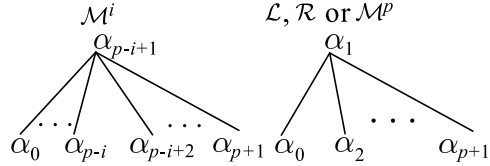}
\caption{For each of the caret types given above: $\m^1$,\mi\
for $i=2,...,p-1$, and \lft,\rt, or $\m^p$, the order of the
nodes of the caret is defined so that for arbitrary nodes $a$
and $b$ with vertex index numbers $\alpha_j$ and $\alpha_k$,
$a<b$ if and only if $j<k$.} \label{type2}
\end{figure}

\par Within a
tree-pair diagram, the carets in the negative and positive
trees with the same index number are paired together and
referred to as a caret pair.  The caret pair with index number
$i$ is called the {\it $ith$ caret pair}, and is denoted by
$\wedge_i$.

\begin{nota}[$\wedge_i$]\label{wedgei}
We use the notation $\wedge_i$ to represent both a single caret
with index number $i$ and to represent the $ith$ caret pair in
a tree-pair diagram; when we use this notation, which of these is meant should be clear from the context.
\end{nota}

\subsection{Final classification of caret
types}\label{alltypes}
\par The following definitions will further refine
our categories of caret types so that we can finally proceed to
the metric.

\begin{defn}[successor, predecessor]\label{sucpred} For two
carets $\wedge_i$ and $\wedge_j$ in a tree, we say that
$\wedge_i$ is a {\it successor} of $\wedge_j$ whenever $i>j$,
and we say that $\wedge_i$ is a {\it predecessor} of $\wedge_j$
whenever $i<j$.
\end{defn}

\begin{rmk}[ancestor/descendant vs. successor/predecessor]
We must not confuse successors with children (or descendants)
and predecessors with parents (or ancestors).  $\wedge_B$ is a
child of $\wedge_A$ if and only if the parent vertex of
$\wedge_B$ is a child vertex of $\wedge_A$, but $\wedge_B$ is a
successor of $\wedge_A$ if and only if $B>A$.  The properties
of being a child or successor of some fixed caret are wholly
independent.  For example, in Figure \ref{exampletree}, in
$T_+$ $\wedge_1$ is a child but not a successor of $\wedge_3$,
and in $T_-$, $\wedge_8$ is a successor but not a child of
$\wedge_6$; in contrast, in $T_-$, $\wedge_7$ is both a child
and a successor of $\wedge_5$.
\end{rmk}

\begin{defn}[leftmost caret]
When we refer to a caret as the leftmost caret with some
property $X$, we mean precisely the caret with property $X$
whose index number is smallest.  So, for example, the leftmost
child of $\wedge_i$ would be the child of $\wedge_i$ with the
smallest index number and the leftmost child successor would be
the caret with the smallest index number which is both a child and a successor of $\wedge_i$.
\end{defn}

\par And now we enumerate the final set of categories of caret
type:

\begin{enumerate}
\item \lftnot.  This is the first and leftmost caret of the
tree. There is one and only one caret of this type in
any non-empty tree.
\item \lftl.  Any left caret not of
type \lftnot\ is of this type.
\item \rtnot.  This is
any right caret for which all successor carets are
right carets.  For example, in Figure
\ref{exampletree}, $\wedge_8\in T_-$ is the only caret
of type \rtnot.
\item \rtr.  This is a right caret
whose immediate successor is a right caret, but which
has at least one successor which is not a right caret.
For example, in Figure \ref{seesawfig},
$\wedge_{m+2}\in S_+$ is of type \rtr\ because its
immediate successor is $\wedge_{m+3}$, which is type
\rt, but its successor $\wedge_{m+np+n}$ is not type
\rt.
\item \rtj.  This is a right caret whose immediate
successor is not a right caret and whose leftmost child
successor is type \mj\ when $j<p$, or \rt\ when
$j=p$.  For example, in $T_+$ in Figure
\ref{exampletree}, the leftmost child successor of
$\wedge_5$ is $\wedge_6$; since $\wedge_6$ is type
$\m^4$, $\wedge_5$ is type $\rt_4$.  A
caret of type \rtp\ can be seen in $T_-$: $\wedge_4$
has as its immediate successor $\wedge_5$, which is not
a right caret, and the leftmost child successor of
$\wedge_4$ is $\wedge_8$, which is type \rt, so
$\wedge_4$ is type \rtp.
\item \minot.  This is a middle caret of type \mi\ that has
    no child successor carets.  For example, in Figure
    \ref{exampletree}, the only carets of type \minot\ for
    some $i\in\{1,...,p\}$ are: $\wedge_3\in T_-$ is type
    $\m^1_\emptyset$, $\wedge_6, \wedge_7\in T_-$ are type
    $\m^3_\emptyset$, $\wedge_2\in T_+$ is type
    $\m^p_\emptyset$, $\wedge_1, \wedge_3\in T_+$ are type
    $\m^2_\emptyset$, $\wedge_4\in T_+$ is type
    $\m^1_\emptyset$, $\wedge_8\in T_+$ is type
    $\m^3_\emptyset$. \item \mij. This is a middle caret of
    type \mi\ with leftmost child successor of type \mj
    (where we will always have $j\leq\ i$).  For example,
    in Figure \ref{exampletree}, $\wedge_5\in T_-$ is type
    $\m^p_3$, $\wedge_6\in T_+$ is type $\m^4_3$, and
    $\wedge_7\in T_+$ is type $\m^3_3$.
\end{enumerate}

\subsection{The metric}
\par We now describe the metric developed by Fordham in \cite
{length} for geodesic length
in $F(p+1)$ with respect to $X$.
According to this metric, each caret pair in the minimal
tree-pair diagram representative of an element of $F(p+1)$
contributes a ``weight" which, when summed over all caret pairs
in the diagram, yields the length of the element in $F(p+1)$.

\begin{nota}[$|w|$]\label{lengthsym}
For given $w\in F(p+1)$, $|w|$
is the length of $w$ w.r.t. $X$.
\end{nota}

\par The {\it weight} of a caret pair in a minimal tree-pair
diagram representing $w\in F(p+1)$ is the
contribution of that caret pair to the length of $w$ (see Table \ref{weighttable}).  The weight depends
upon the type of each caret in the pair and is derived from the
cardinality of the set of
generators which is required to produce the caret pair.

\begin{nota}[$w_{(T_-,T_+)}(\wedge_i)$,
$w_{(T_-,T_+)}(\tau_1,\tau_2)$]\label{weightsym}
If the types of the negative and positive carets in the $ith$
caret pair of $(T_-,T_+)$ are denoted by $\tau_1$ and $\tau_2$
respectively, then we denote the weight of $\wedge_i$ by
$w_{(T_-,T_+)}(\wedge_i)$ or $w_{(T_-,T_+)}(\tau_1,\tau_2)$.
When the tree-pair diagram itself is obvious from the context,
we will often omit the subscript.
\end{nota}

\begin{rmk}
Since Table \ref{weighttable} is symmetric,
$w(\tau_1,\tau_2)=w(\tau_2,\tau_1)$ for all $\tau_1, \tau_2$.
\end{rmk}

\begin{thm}[Fordham \cite {length}, Theorem 2.0.11]
\label{lengththm}Given an element $w=(T_-, T_+)$ in $F(p+1)$,
$|w|$ is the sum of the weights given in Table
\ref{weighttable} for each of the pairs
of carets in $(T_-, T_+)$.  (Note that since only $\wedge_0$ is of type \lftnot, (\lftnot,\lftnot) is the
only possible pairing.)
\end{thm}

\begin {table}[t]
\centering \caption{Weight of types of caret pairs in a
$(p+1)$--ary tree-pair diagram:}
\begin{tabular}{clcccclll}
  \hline
  $(\hbox{ },\hbox{ })$ & \vline & \lftnot & \lftl\ & \rtnot\ &
\rtr\ & \rtj\ & \mlnot\ & $\m^t_u$ \\
  \hline
  \lftnot\ & \vline & 0 & -- & -- & -- & -- & -- & -- \\
  \lftl\ & \vline & -- & 2 & 1 & 1 & 1 & 2 & 2 \\
  \rtnot\ & \vline & -- & 1 & 0 & 2 & 2 & 1 & 3 \\
  \rtr\ & \vline & -- & 1 & 2 & 2 & 2 & 1 & 3 \\
  \rti\ & \vline & -- & 1 & 2 & 2 & 2 & $\displaystyle^{1
\hbox{ for } i\le l}_{3 \hbox{ for } i>l}$ & 3 \\
  \mknot\ & \vline & -- & 2 & 1 & 1 & $\displaystyle^{1 \hbox{
for } j\le k}_{3 \hbox{ for } j>k}$ & 2 & $\displaystyle^{2
\hbox{ for } k\le u}_{4 \hbox{ for } k>u}$ \\
  \mrs\ & \vline & -- & 2 & 3 & 3 & 3 & $\displaystyle^{2
\hbox{ for } l\le s}_{4 \hbox{ for } l>s}$ & 4 \\
  \hline
\end{tabular}
\label{weighttable}
\end{table}

\subsection{How generators act on caret type
pairings}\label{action}
\par Our approach in this paper involves thinking of
multiplication on the right by a generator as an ``action" on a
tree-pair diagram.  When we multiply $x=(T_-,T_+)$ of $F(p+1)$
on the right by $y$, we view $((Ty)_-, (Ty)_+$ as the results
of this ``action" of $y$ on $(T_-,T_+)$.  Diagrams depicting
this ``action" of $g\in X\cup X^{-1}$ on an arbitrary $S_-$ can
be seen in Figure \ref{genaction}.

\begin{figure}[t]
\centering
\includegraphics[width=3.5in]{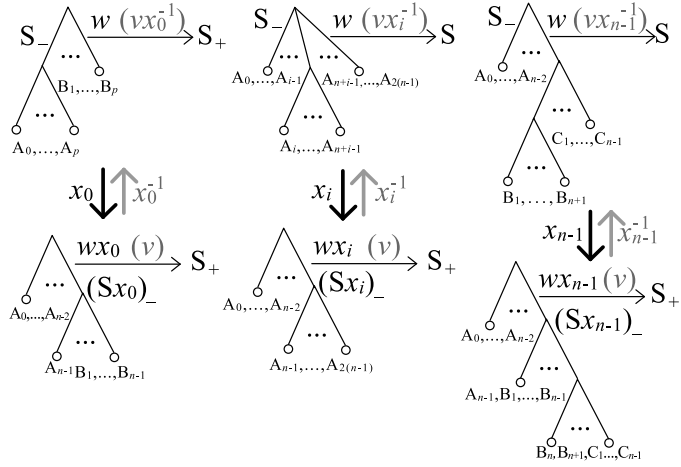}
\caption{The ``action" of given $g\in X\cup X^{-1}$ on
an arbitrary $(p+1)$--ary tree-pair diagram, where we assume
that the tree-pair diagram $(S_-,S_+)$ has already had any
carets added which are needed in order to compute the product.
Black arrows/labels indicate the ``action" of $g$ on the tree-pair
diagram representative of an arbitrary word $w$, and grey
arrows/labels indicate the ``action" of $g^{-1}$ on the tree-pair
diagram representative of an arbitrary word $v$ (Here
$i\in\{1,...,p-1\}$).  Because multiplication on the right has
no effect on the positive tree of a tree-pair diagram after all
carets have been added for multiplication,
the ``action" makes no change to the positive trees (see
Remark \ref{noaddedcarets}).  } \label{genaction}
\end{figure}

\par We now define two conditions which will be used in the
theorems that follow.

\begin{defn}[subtree condition]\label{subcond}
For  fixed $w=(T_-,T_+)\in F(p+1)$, $g\in X\cup
X^{-1}$, $w$ and $g$ fulfil the subtree condition
when $wg$ can be computed without adding carets.
\end{defn}

\begin{defn}[minimality condition]\label{mincond}
For fixed $w=(T_-,T_+)\in F(p+1)$, $g\in X\cup
X^{-1}$, $w$ and $g$ fulfil the minimality
condition when $((Tg)_-,(Tg)_+)$ is minimal.
\end{defn}

\par Fordham proves that when these two conditions are met,
only one caret pair in the tree-pair diagram changes type as a
result of the ``action" of $g$:

\begin{thm}[Fordham \cite{length}, Theorem
2.1.1]\label{onecaret}
If $w=(T_-,T_+)\in F(p+1)$ and $g\in X\cup X^{-1}$ satisfy the
subtree and minimality conditions, then there is exactly one
caret $\wedge _i$ in the tree-pair diagram that changes type
under the multiplication $wg${\rm ;} that is, if we let $\tau
_{T_-}(\wedge_i)$ denote the caret type of $\wedge_i$ in $T_-$
in the tree-pair diagram $(T_-,T_+)$, then $\exists\ i$ such
that
$$\tau _{T_-}(\wedge _i)\ne \tau _{(Tg)_-}(\wedge _i)
\hbox{ and } \tau _{T_-}(\wedge _j)=\tau _{(Tg)_-}(\wedge _j)
\forall\ j\ne i$$
\end{thm}

The caret $\wedge_i$ which changes type when the conditions of
Theorem \ref{onecaret} are met will always be in the negative
tree:

\begin{rmk} \label{noaddedcarets}
When multiplying an element $x=(T_-,T_+)$ in $F(p+1)$ by an
element $y$ on the right, if the subtree condition is met, then
the type of caret $\wedge_i$ is the same in both $T_+$ and
$(Ty)_+$ for all caret index numbers $i$.  The type of
$\wedge_i$ will be different in $(Ty)_+'$ than in $T_+$ only if
the minimality condition is not met.
\end{rmk}

When either the subtree or minimality condition fails, we have
an alternate theorem which can help us to determine the effect
of multiplication on an element's length without computing it
directly:

\begin{thm}[Fordham \cite {length}, Theorems 2.1.3 and
2.14]\label{addcaret}  If $g\in X\cup X^{-1}$ and
$w=(T_-,T_+)\in F(p+1)$, do not fulfil:
\begin{enumerate}
\item the subtree condition when computing $wg$, then
$|wg|>|w|$. \item the minimality condition when
computing $wg$, then $|wg|=|w|-1$.
\end{enumerate}
\end{thm}

\section{Seesaw words with arbitrary swing exist in $F(p+1)$}
\subsection{Seesaw words in $F(p+1)$}
\begin{thm}\label{seesawNF}
Any word in $F(p+1)$ with the following normal form, where
$m,n\in\mathbb{N}$ is a seesaw word with respect to $x_0$ in
$X$.

\begin{eqnarray*}
x_0^{m-1}x_px_{np^2+(m+n)p}\left(\overset{pn} {\underset{i=1}
\prod} x^{-1}_{np^2+(m+n-i+1)p-i}\right) x_0^{-m}
\end{eqnarray*}

The minimal tree-pair diagram representative of an element of
this form can be seen in Figure \ref{seesawfig}.  This family of seesaw words will be denoted ${\mathcal
S}$.
\end{thm}

\begin{figure}[t]
\centering
\includegraphics[width=\maxwidth]{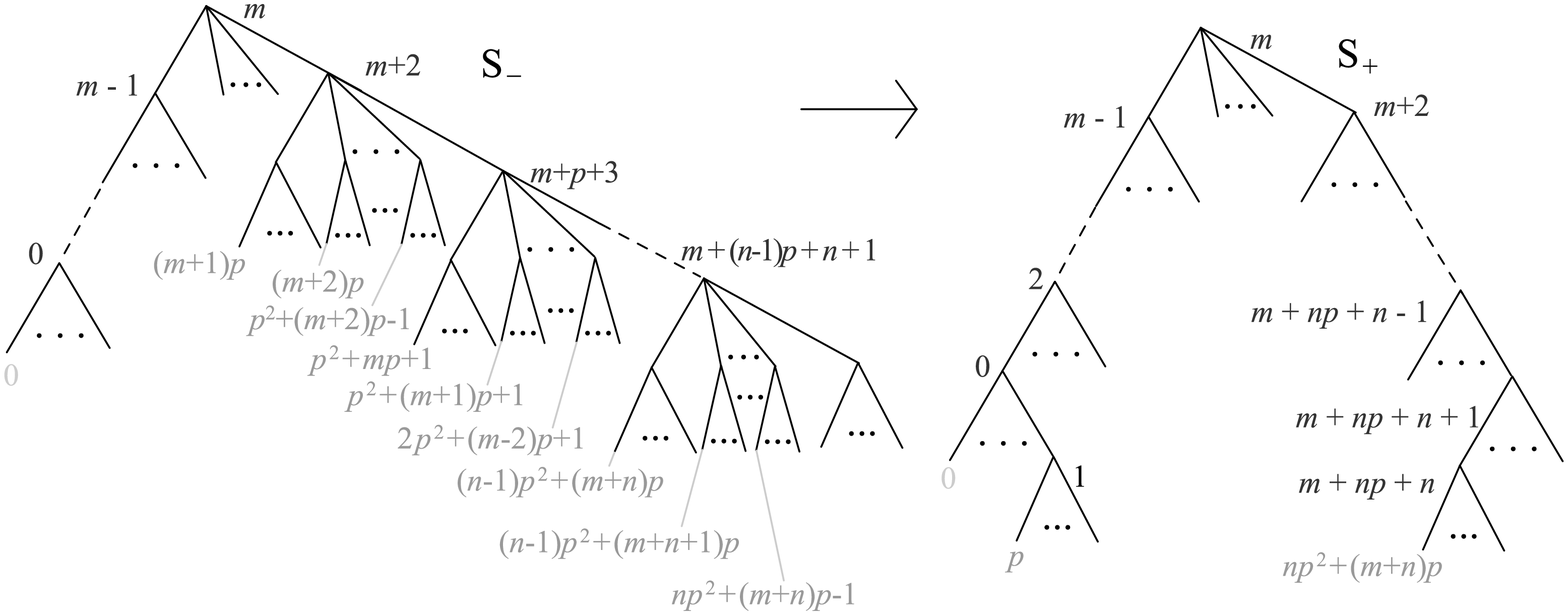}
\caption{Minimal tree-pair diagram representative of an
arbitrary seesaw element in the family ${\mathcal
S}$.  The letter $m$ denotes the
number of carets of type \lftl\ in $S_-$ and the letter $n$
denotes the number of carets of type \rt\ on the right side of
$S_-$ which are not of type \rtnot.
}\label{seesawfig}
\end{figure}

\par The proofs that follow will be concerned entirely with
showing that all elements with minimal tree-pair diagram
representative of the form given in Figure \ref{seesawfig} are
seesaw words with respect to $x_0$.  The algebraic expression
is entirely determined by the minimal tree-pair diagram; to see
how this algebraic expression can be obtained from the
tree-pair diagram given in Figure \ref{seesawfig}, see the
section on normal forms of $F(p+1)$ in \cite{notMAC}.  This family ${\mathcal S}$ is a generalization of the family of
seesaw words introduced by Cleary and Taback in \cite{seesaw}.

\par For our proof, we take arbitrary $w=(S_-,S_+)\in{\mathcal S}$.  First we prove
that $w$ satisfies part 1 of the definition of seesaw words with respect to
$x_0\in X$.

\begin{lem}\label{shrinking}
$$|wx_0^{\pm q}|=|w|- q\hbox{ for all }q\hbox{ such that
}0<q<m-1,n-1$$
where $m$ denotes the number of carets of type
\lftl\ in $S_-$ and $n$ denotes the number of carets
of type \rt\ on the right side of $S_-$ which are not type
\rtnot.
\end{lem}

\begin{proof}
\par We prove this by induction.  Throughout this proof, we let
$(S^q_-,S^q_+)$ denote $((Sx^{-q})_-,(Sx^{-q})_+)$ and we let
$(R^q_-,R^q_+)$ denote $((Sx^{q})_-,(Sx^{q})_+)$, where $q>0$
in both cases.  Our inductive hypothesis assumes that
$wx_0^{q}$ and $wx_0^{-q}$ have minimal tree-pair diagram
representatives of the form given in Figures \ref{wx0qinv} and
\ref{wx0q} respectively.

\begin{figure}[t]
\centering
\includegraphics[width=4in]{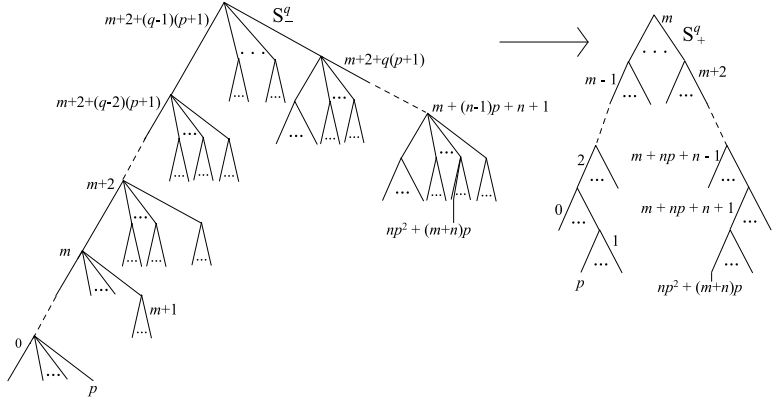}
\caption{Minimal tree-pair diagram representative of
$wx_0^{-q}$ (when $0<q<n-1$) for $w\in{\mathcal
S}$.}\label{wx0qinv}
\end{figure}

\begin{figure}[t]
\centering
\includegraphics[width=3.5in]{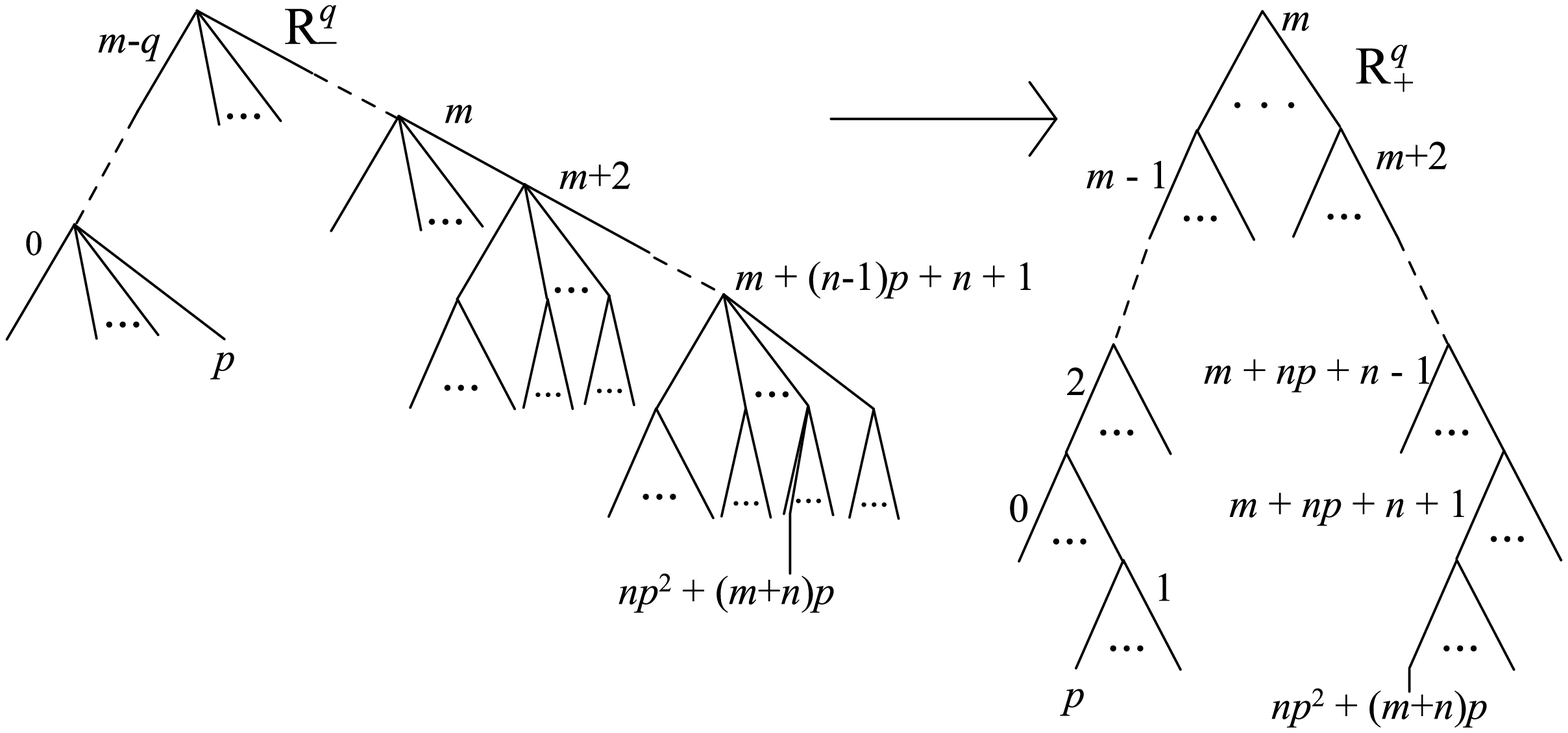}
\caption{Minimal tree-pair diagram representative of $wx_0^{q}$
(when $0\leq q<m-1$) for $w\in{\mathcal S}$.}\label{wx0q}
\end{figure}

\begin{enumerate}
\item $|wx_0^{-q}|$:  We begin by considering the case when
    $q=1$. Performing the multiplication $wx_0^{-1}$ using
    the minimal tree-pair diagram representatives of $w$
    and $x_0^{-1}$ in Figures \ref{seesawfig} and
    \ref{geninf} respectively, we obtain Figure
    \ref{wx0qinv} (when $q=1$); $(S^1_-,S^1_+)$ is minimal
    because there are only two exposed carets in $S^1_+$:
    the carets with leftmost leaf index numbers $p$ and
    $np^2+(m+n)p$, but neither of the leaves with these
    index numbers in $S^1_-$ is the leftmost leaf of an
    exposed caret.

    \par Our inductive hypothesis will be that
    $|wx_0^{-q}|=|w|-q$ for some $q$ such that
    $0<q<m-1,n-1$ and that $wx_0^{-q}$ has minimal
    tree-pair diagram representative $(S^q_-,S^q_+)$ (see
    Figure \ref{wx0qinv}).  Now we assume our hypotheses
    hold for some $q=j-1$
such that $0<j<n-2$ and we consider what happens when
we multiply $wx_0^{-(j-1)}$ by $x_0^{-1}$ on the right.
By our inductive hypothesis, the tree-pair diagram in
Figure \ref{wx0qinv} is the minimal
representative of $wx_0^{-(j-1)}$ when $q=j-1$.
Because $wx_0^{-(j-1)}$ and $x_0^{-1}$ satisfy the
subtree condition, the positive tree $S^{j-1}_+$
remains unchanged after multiplication by $x_0^{-1}$
(see Remark \ref{noaddedcarets}).  So we consider which
changes $x_0^{-1}$ makes to the
negative tree.

    \par By looking at Figures \ref{genaction} and
\ref{wx0qinv} which represent $wx_0^{-(j-1)}$ and
$x_0^{-1}$ respectively, we can see that
multiplying $wx_0^{-(j-1)}$ by \xinv0\ changes $\wedge
_{m+2+(j-1)(p+1)}$ (the rightmost child of the root) in
$S_-^{j-1}$ from type \rtone\ to type \lftl. This is the
only change in the negative tree.  So we can see that the
resulting tree-pair diagram representative for $wx_0^{-j}$
will have $\wedge_{m+2+(j-1)(p+1)}$ as the root caret and
$\wedge_{m+2+(j-2)(p+1)}$ as the leftmost child of the
root.  The relative location of all other carets in the
tree will be identical to their placement in the minimal
tree-pair diagram representative for $wx_0^{-(j-1)}$.  So
it is clear that Figure \ref{wx0qinv} (when $q=j$) is a
tree-pair diagram representative for $wx_0^{-j}$.  Now we
need only show that it is minimal; we note that any carets which are exposed
in $(S^j_-,S^j_+)$ would also have
been exposed in $(S^{j-1}_-,S^{j-1}_+)$, so minimality of $(S_-^{j-1},S_+^{j-1})$ implies minimality of $(S_-^{j},S_+^{j})$.

\par Now we consider the effect of multiplication of
$wx_0^{-(j-1)}$ by $x_0^{-1}$ on the length of $wx_0^{-(j-1)}$.
The caret $\wedge_{m+2+(j-1)(p+1)}$ will
always be a successor of the caret $\wedge_m$ in both
$S_+^{j-1}$ and
$S_+^j$, and the only successors of $\wedge_m$ in
$S_+^{j-1}$ and $S_+^j$ which are not of type \rtr\ are
$\wedge_{m+np+n}$ and $\wedge_{m+np+n+1}$.  Since $j<n$, it
is clear that $m+2+(j-1)(p+1)<m+np+n$ and therefore
$\wedge_{m+2+(j-1)(p+1)}$ is of type \rtr\ in $S^{j-1}_+$
and $S_+^j$. Therefore, this change in the caret
$\wedge_{m+2+(j-1)(p+1)}$ in the negative tree from type
\rtone\ to type \lftl\ changes the pairing from
(\rtone,\rtr), which has weight 2, to (\lftl,\rtr), which
has weight 1 (see Table \ref{weighttable}).  So $|wx_0^{-j}|=|wx_0^{-(j-1)}|-1$.  And since
by our inductive hypotheses $|wx_0^{-(j-1)}|=|w|-(j-1)$,
$$|wx_0^{-j}|=|w|-(j-1)-1=|w|-j \hbox{ for all }j \hbox{ s.t. }
0<j<n-1$$

\item $|wx_0^q|$:  The proof that $|wx_0^q|=|w|-q$ is
similar to the proof that $|wx_0^{-q}|=|w|-q$.  The
primary difference is that the caret in
$(R_-^{j-1},R_+^{j-1})$ in Figure \ref{wx0q} whose type
is changed by multiplication by $x_0$ is $\wedge
_{m-(j-1)}$ (the root caret) in $R_-^{j-1}$, which is
changed from type \lftl\ to type \rtr\ (or \rtp\ in the
case $j=1$). In the same way as for the $x_0^{-1}$
case, this leads to the conclusion that Figure
\ref{wx0q} is a minimal tree-pair diagram
representative of $wx_0^{j}$ when $q=j$.  Then to
compute the effect of multiplication by $x_0$ on
length, we note that the caret $\wedge _{m-(j-1)}$ in
$R_+^{j-1}$ or $R_+^j$ will always be of type \lftl\
for any given $j=1,\dots,m-2$ because
$\wedge_{m-(j-1)}$ is a predecessor of the root
$\wedge_m$ in $R_+^{j-1}$ and $R_+^j$ since $m-(j-1)<m$and,
and the only predecessors of the root in $R_+^{j-1}$ or
$R_+^j$ which are not of type \lftl\ are $\wedge_1$ and
$\wedge_0$.  Since $j\le m-2$ guarantees that
$m-(j-1)>1$ for all possible $j$,
$\wedge_{m-(j-1)}\ne\wedge_1$ or $\wedge_0$. Therefore,
this change in the caret $\wedge _{m-(j-1)}$ from type
\lftl\ to type \rtr\ (or \rtp\ when $j=1$) changes the
pairing from (\lftl,\lftl), which has weight 2, to
(\rtr,\lftl) (or (\rtp,\lftl) when $j=1$), which has
weight 1 (see Table \ref{weighttable}).  Then similarly
to the $x_0^{-1}$ case, we can use induction to
conclude that $|wx_0^{q}|=|w|-q$ and that Figure
\ref{wx0q} is a minimal tree-pair diagram
representative of $wx_0^{q}$ for all $q$ such that
$0\le q<m-1$.
\end{enumerate}
\end{proof}

\par Now we show that all $w\in\mathcal{S}$ satisfy part 2 of
the
definition of a seesaw word by considering the ``action" of each $g\in X\cup X^{-1}$ on
$wx_0^{\pm q}$ for arbitrary $q$ such that $0\le q<m-1,n-1$,
and showing that this ``action" always results in increased
length.

\begin{lem}\label{wx0pmg}
For $w\in{\mathcal S}$, $\epsilon\in\{-1,1\}$, and arbitrary
$q$ s.t. $0<|q|<m-1,n-1$, $$|wx_0^{\epsilon
q}g|\ge|wx_0^{\epsilon q}|$$ for all $g\in X\cup X^{-1}$.
\end{lem}

\begin{proof}We consider each possible combination of values of
$\epsilon$ and $g$:
\begin{enumerate}
\item $|wx_0^{-q}x_i^{\pm1}|$, $i\in\{1,2,...,p\}$: First
we note that $wx_0^{-q}$ and $x_i^{\pm1}$ when $0\le
q<m-1,n-1$ and $i=1,2,...,p$ satisfy both the subtree
and minimality conditions of Theorem \ref{onecaret}
except when $q=0$ and $i=1,...,p-1$.  So only one caret
will change type in the negative tree and the positive
tree will remain unchanged after multiplication in
these cases.

 \par We begin with the case $q=0$.
 \begin{enumerate}
\item $|wx_i^{-1}|$: Multiplying $w$ by $x_i^{-1}$ changes
$\wedge_{m+2}$ from type \rtone\ to type \mione\ and changes no other caret types.  Since
all the carets in $S_+$ and $S_+^1$ which succeed
$\wedge_m$ and precede $\wedge_{m+np+n}$ have type \rtr\
and $m<m+2<m+np+n$, $\wedge_{m+2}$ is of type \rtr\ in
$S_+$ and $S_+^1$.  So the change in the type pair of
$\wedge_{m+2}$ goes from (\rtone,\rtr) which has weight 2
to type (\mione,\rtr) which has weight 3, and clearly
$|wx_i^{-1}|>|w|$.
\item $|wx_i|$: Multiplying $w$ by $x_i$ when
$i=1,...,p-1$ does not satisfy the subtree condition and
therefore by Theorem \ref{addcaret}, $|wx_i|>|w|$.
Multiplying $w$ by $x_p$ changes $\wedge_{m+1}$ from type
\mpnot\ to type \rtr\ and changes no other caret types.  Since $m<m+1<m+np+n$,
$\wedge_{m+1}$ is of type \rtr\ in $S_+$ and $R_+^1$.  So
this change in the type pair of $\wedge_{m+1}$ goes from
(\mpnot,\rtr) which has weight 1 to (\rtr,\rtr) which has
weight 2, so $|wx_p|>|w|$.
\end{enumerate}

    \par Now we consider multiplying $wx_0^{-q}$ for
$0<q<m-1,n-1$ by
$x_i^{\pm1}$ for $i=1,2,...,p$, when both conditions of Theorem \ref{onecaret} are met.
\begin{enumerate}
\item $|wx_0^{-q}x_i^{-1}|$:  Multiplying $wx_0^{-q}$
by \xinv{i} changes $\wedge _{m+2+q(p-1)}$ (the
right child of the root) in $S_-^q$ from type
\rtone\ to type \mione. In $S_+^q$ and $S_+^{q+1}$,
all carets which succeed $\wedge_m$ and precede
$\wedge_{m+np+n}$ have type \rtr, so since
$m<m+2+q(p-1)<m+np+n$ (because $q<n-1$), $\wedge
_{m+2+q(p-1)}$ is of type \rtr\ in $S_+^q$ and
$S_+^{q+1}$.  So this multiplication changes the
type pair of $\wedge_{m+2+q(p-1)}$ from
(\rtone,\rtr), which has weight 2, to
(\mione,\rtr), which has weight 3. So
$|wx_0^{-q}\x{i}^{-1} |=|w|-q+1$.

\item $|wx_0^{-q}x_i|$:  Multiplying $wx_0^{-q}$ by
\x{i}\ changes $\wedge _{m+2+(q-1)(p-1)+i}$  (the
$ith$ child of the root) in $S_-^q$ from type
\minot\ to type \riplusone\ when $i<p$ and to type
\rtr\ when $i=p$.  Again, since
$m<m+2+(q-1)(p-1)+i<m+np+n$ (because $q<n-1$),
$\wedge_{m+2+(q-1)(p-1)+i}$ is of type \rtr\ in
$R_+^q$ and $R_+^{q+1}$.  So this multiplication
changes the type pair of
$\wedge_{m+2+(q-1)(p-1)+i}$ from (\minot,\rtr),
which has weight 1, to (\riplusone,\rtr) when $i<p$
and (\rtr,\rtr) when $i=p$, both of which have
weight 2. So $|wx_0^{-q}\x{i} |=|w|-q+1$.
\end{enumerate}

\item $|wx_0^{q}x_i^{\pm1}|$, $i\in\{1,2,...,p\}$: Now we
consider multiplying $wx_0^{q}$ for $0<q<m-1,n-1$ by
$x_i^{\pm1}$ for $i=1,2,...,p$.  First we note that
$wx_0^{q}$ and $x_i^{-1}$ when $0\le q<m-1,n-1$ and
$i=1,2,...,p$ satisfy both the subtree and minimality
conditions of Theorem \ref{onecaret}.  So only one
caret will change type in the negative tree and the
positive tree will remain unchanged after
multiplication in this case.
\begin{enumerate}
\item $|wx_0^{q}x_i^{-1}|$: If we let $i=1,\dots,p$,
multiplying $wx_0^{q}$ by \xinv{i} changes the
rightmost child of the root, which is
$\wedge_{m-q+1}$ when $q>0$ and $\wedge_{m+2}$ when
$q=0$.  When $q=0$, we can conclude that
$\wedge_{m+2}$ is of type \rtr\ in both $S_+$ and
$R_+^1$ (since $m<m+2<m+np+n$), and $\wedge_{m+2}$
is changed from type \rtone\ to type \mione,
changing the type pairing from (\rtone,\rtr) which
has weight 2 to (\mione,\rtr) which has weight 3.
When $q>0$, we can conclude that $\wedge_{m-q+1}$
is of type \lftl\ in both $R_+^q$ and $R_+^{q+1}$
since all carets which succeed $\wedge_1$ and
precede $\wedge_{m+1}$ in $R_+^q$ and $R_+^{q+1}$
are of type \lftl\ and clearly $1<m-q+1<m+1$ (since
$q<m-1$).  When $q=1$, $\wedge_{m-q+1}$ is changed
from type \rtr\ to type \minot, changing the type
pairing from (\rtp,\lftl) or (\rtr,\lftl), both of
which have weight 1, to (\minot,\lftl) which has
weight 2.  So $|wx_0^{q}\x{i}^{-1} |=|w|-q+1$.

\item $|wx_0^{q}x_i|$, $i\in\{1,2,...,p\}$:
\begin{enumerate}
\item $|wx_0^{q}x_i|$, $i\in\{1,2,...,p\}$, except
when $q=0$ and $i=p$:   In this case,
multiplying $wx_0^{q}$ by \x{i}\ does not
satisfy the required conditions of Theorem
\ref{onecaret} because we must add a caret
before we can complete the multiplication, so
we know from Theorem \ref{addcaret} that
$|wx_0^{q}\x{i}|>|wx_0^{q}|$ in this case.
\item
$|wx_0^{q}x_p|$:  When $q=0$, $wx_0^q$ and
$x_p$ satisfy the required subtree and
minimality conditions of Theorem \ref{onecaret}
and therefore only one caret changes type in
the negative tree and the positive tree remains
unchanged.  The caret $\wedge_{m+1}$ is changed
from type \mpnot\ to type \rtr.  Since
$m<m+1<m+np+n$, it is clear that $\wedge_{m+1}$
is of type \rtr\ in $S_+$ and $R_+^1$, and so
the change in type pairing goes from
(\mpnot,\rtr) which has weight 1 to (\rtr,\rtr)
which has weight 2.  So we can conclude that
that $|wx_0^{q}\x{i}|>|wx_0^{q}|$ in this
case.
\end{enumerate}
\end{enumerate}
\end{enumerate}
\end{proof}

\begin{proof}[Proof of Theorem \ref{seesawNF}]
This proof follows immediately from Lemma \ref{shrinking},
Lemma \ref{wx0pmg}, and Definition \ref{seesawword}.  So all
$w\in{\mathcal S}$ are seesaw words, and we can create such
words with any given swing $k$ (where $0<k<min\{m-1,n-1\}$) by
choosing $m$ and $n$ such that $m,n>k+1$.
\end{proof}

\begin{cor}\label{seesawexist}
Thompson's group $F(p+1)$ contains seesaw words of arbitrarily
large swing with respect to $x_0\in X$.
\end{cor}

\subsection{Consequences}
\begin{lem}\label{kfellowtrav}
Given any constant k, there exists a word $w\in {\mathcal S}$
such that no geodesics paths from the identity to $wx_0$, $w$,
or $wx_0^{-1}$ satisfy the k-fellow traveler property.
\end{lem}
\begin{proof}
This holds for the same reasons that
Prop. 4.2 in \cite{seesaw} holds for $p=1$.
\end{proof}

\begin{thm}\label{notcombable}
Thompson's group $F(p+1)$ is not combable by geodesics.
\end{thm}
\begin{proof}
This holds for the same reasons that
Theorem 4.2 in \cite{seesaw} holds for $p=1$.
\end{proof}

\begin{thm}[Theorem 30 in
\cite{geolang}]\label{noreglang}
A group $G$ generated by a finite set $X$ with
seesaw elements of arbitrary swing w.r.t. $X$ has no
regular language of geodesics.
\end{thm}

\begin{cor}\label{notreg}
There does not exist a regular language of geodesics for $F(p+1)$ with respect to $X$.
\end{cor}

\section{Dead ends exist in Thompson's group $F(p+1)$}
\par Cleary and Taback in \cite{combpropF} have shown that
$F(2)$ has dead ends, and that all these dead ends have depth
2.  In this section we use a similar approach to extend their
results to $F(p+1)$ for all $p\in\mathbb{N}$.

\subsection{Dead ends in $F(p+1)$}
\par The proofs in this section will contain many tree-pair
diagrams which use the following notational convention.

\begin{nota}[Subtrees in tree-pair
diagrams]\label{subtreecircles}
When depicting tree-pair diagrams, the symbol \unitlength=1pt
\begin{picture}(8,8)\put(5,4){\circle*{8}}\end{picture}
indicates the presence of a non-empty subtree, and the the
symbol
\unitlength=1pt\begin{picture}(8,8)\put(5,4){\circle{8}}\end{picture}
indicates the presence of a (possibly empty) subtree.  When
neither of these symbols are used, it is assumed that there is
no subtree present.
\end{nota}

\par Now we proceed to show that elements of
$F(p+1)$ are dead ends if and only if they have a minimal tree-pair
diagram representative with a specific form.

\begin{thm} \label{deadenddiag}
All dead ends in $F(p+1)$ under $X$ have minimal tree-pair diagrams of the form given in Figure
\ref{deadend}.
\end{thm}

\begin{figure}[t]
\centering
\includegraphics[width=4in]{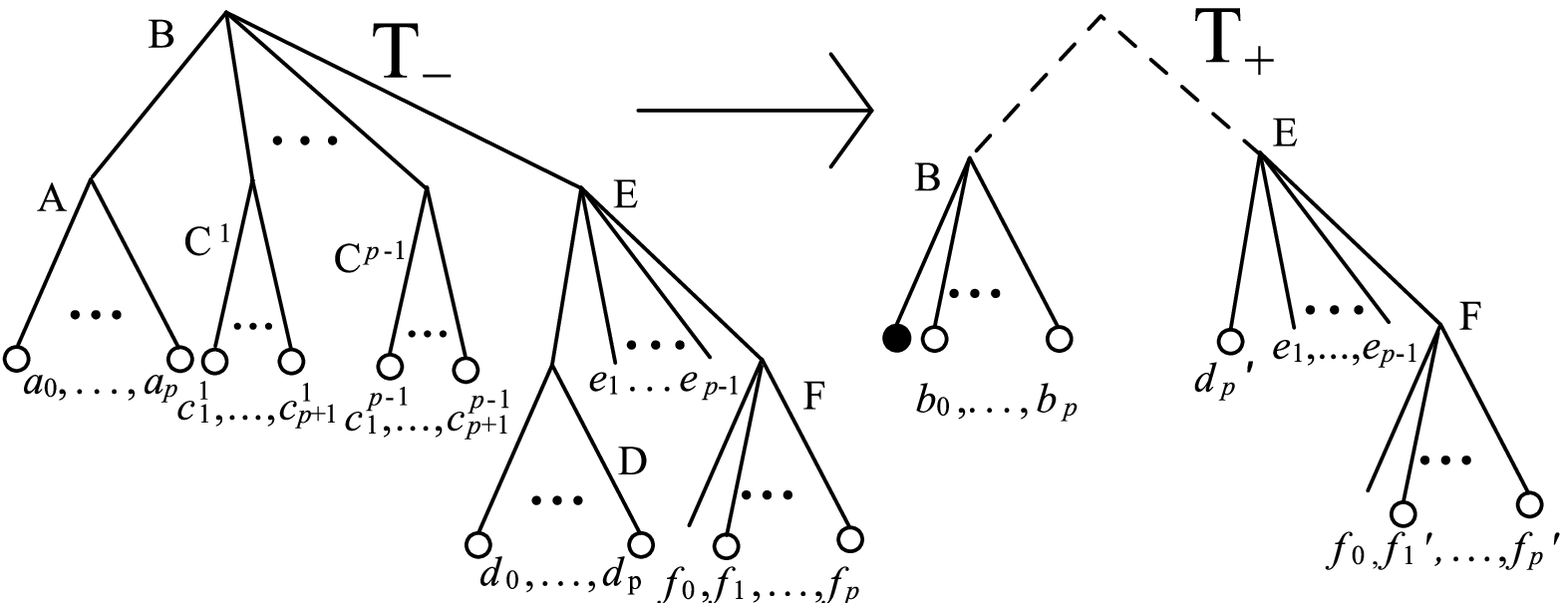}
\caption{{\bf Form of Minimal Tree-pair Diagram for All Dead
Ends in ${\bf F(p+1)}$}.  If $\wedge_{C^i}$, for some
$i\in\{1,...,p-1\}$, is of type \minot\ in $T_-$, then
$\wedge_{C^i}$ must be of type \lftl, \rtk\ \mlnot\ or \mlk\ in
$T_+$, where $k,l\le i$; similarly, if $\wedge_D$ is of type
\mpnot\ in $T_-$, then $\wedge_D$ cannot be of type \rtr\ or
\rtnot\ in $T_+$. }\label{deadend}
\end{figure}

\par We note that in Theorem \ref{deadenddiag} we mean that the
minimal form of the dead end tree-pair diagram representative
must include all of the carets explicitly given in Figure
\ref{deadend}, so, for example, at least one of the subtrees
labeled $f_1,...,f_p$ in $T_-$ and at least one of the subtrees
labeled $f_1',...,f_p'$ in $T_+$ are non-empty because
otherwise $\wedge_F$ would cancel.  The proof of this theorem is based upon recognizing how
the ``action" of each $g\in X\cup X^{-1}$ affects an arbitrary
tree-pair diagram $(T_-,T_+)$.

\begin{rmk}\label{caretchange}
The negative tree of any $(p+1)$--ary tree-pair diagram can be
written in the (possibly non-minimal) form given by Figure
\ref{negtree}, and for any negative tree in this form, the
``action" of any $g\in X\cup X^{-1}$ on $T_-$ will change only
one caret type in that tree (This is because the only other
changes in type that can occur when multiplying by a generator
are caused by the addition of carets to the tree-pair diagram,
but by definition, negative trees in this form will belong to
tree-pair diagrams to which all carets needed in order to
multiply by a generator or its inverse have already been added
- see Theorem \ref{onecaret} and Remark \ref{noaddedcarets}).

\par The ``action" of $g$ on this negative tree will produce the following
caret type change (see Figure \ref{genaction}):
\begin{enumerate}
\item $x_0$ takes the type of $\wedge_B$ from \lftl\ to
\rt.
\item $x_0^{-1}$ takes the type of $\wedge_E$ from
\rt\ to \lftl.
\item $x_i$ for $i=1,...,p-1$ takes the
type of $\wedge_{C^i}$ from \mi\ to \rt.
\item
$x_i^{-1}$ for $i=1,...,p-1$ takes the type of
$\wedge_E$ from \rt\ to \mi.
\item $x_{p}$ takes the
type of $\wedge_D$ from \mpp\ to \rt.
\item
$x_{p}^{-1}$ takes the type of $\wedge_E$ from \rt\ to
\mpp.
\end{enumerate}
\end{rmk}

\begin{figure}[t]
\centering
\includegraphics[width=2in]{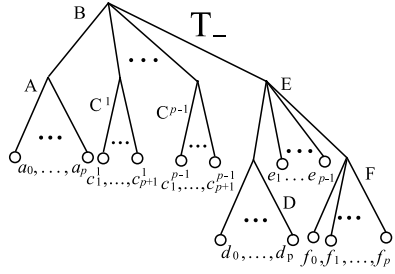}
\caption{A (possibly non-minimal) representation of a negative
tree in an arbitrary $(p+1)$--ary tree-pair diagram}\label{negtree}
\end{figure}

\par Because a dead end $w$ by definition must not increase in
length when multiplied by $g\in X\cup X^{-1}$ (by Theorem
\ref{addcaret}), the product $wg$ must satisfy the subtree
condition for any $g$.

\begin{lem}\label{negtreeform}
All dead ends must have a minimal tree-pair diagram with
negative tree of the form given by Figure \ref{negtree}, and
any dead end $w$ must satisfy the subtree and minimality
conditions with respect to all possible $g\in X\cup X^{-1}$.
\end{lem}

\begin{proof}
A minimal $(p+1)$--ary tree-pair diagram representing an
arbitrary element $x\in F(p+1)$ will have a negative tree of
this form if and only if $x$ and $g\in X\cup X^{-1}$ satisfy
the subtree condition (see Remark \ref{caretchange}).  For an
arbitrary dead end $w$, we cannot have $|wg|>|w|$, so by
Theorem \ref{addcaret}, $w$ must satisfy the subtree conditions
with respect to all possible $g$.

\par The fact that $w=(T_-,T_+)$ satisfies the minimality
condition with respect to all possible $g$ follows directly
from the fact that it satisfies the subtree condition.  The
subtree condition implies that $T_+=(Tg)_+$, and therefore, the
only way in which exposed caret pairs may exist in
$((Tg)_-,(Tg)_+)$, is if the ``action" of $g$ on $(T_-,T_+)$
causes carets to be exposed in $(Tg)_-$ which were not exposed
in $T_-$.  However, if we consider the ``action" of each $g$ on
the negative tree of $w$, which must be of the form given in
Figure \ref{negtree}, we can see that for all $g\in X\cup
X^{-1}$, the only carets which will be exposed in $(Tg)_-$ are
those which are also exposed in $T_-$ (see Figure
\ref{genaction}, or consider Figures \ref{wxnot},
\ref{wxnotinv}, \ref{wxi}, \ref{wxiinv}, \ref{wxp} and
\ref{wxpinv} which follow). Therefore $((Tg)_-,(Tg)_+)$ is
minimal for all $g$.
\end{proof}

\begin{cor}\label{onlyonecaret}
For all dead ends $w=(T_-,T_+)$ and all $g\in X\cup X^{-1}$,
the ``action" of $g$  on $(T_-,T_+)$ only changes the type of
one caret in $T_-$ and leaves the types of all carets in $T_+$
unchanged.
\end{cor}
\begin{proof}
This follows immediately from Lemma \ref{negtreeform}, Remark
\ref{noaddedcarets} and Theorem \ref{onecaret}.
\end{proof}

\par So now we can proceed to prove Theorem \ref{deadenddiag}
by observing which caret changes type in the tree-pair diagram
when each $g$ ``acts" on an arbitrary dead end $w=(T_-,T_+)$ and
then enumerating those conditions which must be met by
$(T_-,T_+)$ in order for this type change to result in a
decrease in length (we note that length cannot remain unchanged
after multiplication by $g$ because in $F(p+1)$ all relators
are of even length).  By showing that these conditions will be
met if and only if $w$ satisfies those conditions laid out in
Theorem \ref{deadenddiag}, we will conclude our proof of the
theorem.  Before continuing with our proof, we first introduce
some notation.

\begin{nota}[$\tau(\wedge_j)$ and
$\Delta_g(\wedge_j)$]\label{wtnot}
$\tau_{T_+}(\wedge_j)$ and $\tau_{(T_-,T_+)}(\wedge_j)$
represent the type of the caret $\wedge_j$ in the tree $T_+$
and the the type pair of the caret pair $\wedge_j$ in the
tree-pair diagram $(T_-,T_+)$, respectively.

$\Delta_g(\wedge_j)$ denotes the change in weight of the caret
pair $\wedge_j$ during multiplication by some $g\in X\cup
X^{-1}$, where the original tree-pair diagram and the resulting
tree-pair diagram should be clear from the context.
\end{nota}

\begin{proof}[Proof of Theorem \ref{deadenddiag}]
\par We consider multiplying our dead end element $w=(T_-,T_+)$
by each
$g\in X\cup X^{-1}$ and enumerate which caret in the negative
tree has its type changed by this multiplication and the effect
of this change on the length of the element (see Table
\ref{weighttable}).

\par For a clearer organizational structure, we organize this
process by the caret in $T_-$ which is affected by the
multiplication.  The labeled carets in $T_-$ are (see Figure
\ref{negtree}): $\wedge_A,\wedge_B\wedge_{C^i}\hbox{ for
}i=1,...,p-1,\wedge_D,\wedge_E,\wedge_F$.  To see which $g$
affects which caret pair in $(T_-,T_+)$, we consult Remark
\ref{caretchange}.

\begin{enumerate}
\item Conditions on $\wedge_A$ in $(T_-,T_+)$: We know from
Remark \ref{caretchange} that there is no $g\in X\cup
X^{-1}$ which will change the type of $\wedge_A$ in the
negative tree, so we have no conditions on the type of
this caret unless they are imposed by the required
types of other carets within the tree.  By definition
$\wedge_A$ is of type \lft\ in $T_-$.  In $T_+$, the
only conditions on $\wedge_A$ will come from the
conditions imposed on $\wedge_B$ (see (\ref{B}));
because $\wedge_B$ in $T_+$ must be of type \lft\ and
since $\wedge_A$ is a predecessor of $\wedge_B$,
$\wedge_A$ in $T_+$ must be of type \lft\ or of type
\m\ with an ancestor of type \lft.

\item \label{B} Conditions on $\wedge_B$ in $(T_-,T_+)$: We
know from Remark \ref{caretchange} that only $x_0$ will
change the type of $\wedge_B$ in the negative tree,
from type \lftl\ to type \rt. If we look
at $(T_-,T_+)$, we can see that in this case we can
compute the types more specifically: $x_0$ will change
the type of $\wedge_B$ in the negative tree from type
\lftl\ to type \rtone\ because $\wedge_B$'s leftmost
child successor is $\wedge_{C^1}$, which is of type
$\m^1$ (see Figure \ref{wxnot}).   Table \ref{wx0table}
lists the change in weight (taken from Table
\ref{weighttable}) of this caret pair for each possible
caret type pair of $\wedge_B$. From this table we
conclude that $\wedge_B$ in $T_+$ must be of type
\lftl\ because this is the only caret pairing in
$(T_-,T_+)$ for $\wedge_B$ which will result in
$|wx_0|<|w|$.

\begin {table}[t]
\centering \caption{How $x_0$ ``acts" on $w(\wedge_B)$ in
arbitrary dead end $w=(T_-,T_+)$, listed by possible types
of $\wedge_B\in T_+$.  Here $\tau_{T_-}(\wedge_B)=\lftl$.}
\begin{tabular}{clccc}
  \hline
  $\tau_{T_+}(\wedge_B)$ & \vline &
$\tau_{(T_-,T_+)}(\wedge_B)$ &
$\tau_{((Tx_0)_-,(Tx_0)_+)}(\wedge_B)$ &
$\Delta_{x_0}(\wedge_B)$ \\
  \hline
  \lftl\ & \vline & (\lftl,\lftl) & (\rtone,\lftl) & -1 \\
  \rtnot\ & \vline & (\lftl,\rtnot) & (\rtone,\rtnot) & 1
\\
  \rtr\ & \vline & (\lftl,\rtr) & (\rtone,\rtr) & 1 \\
  \rtj\ & \vline & (\lftl,\rtj) & (\rtone,\rtj) & 1 \\
  \minot\ & \vline & (\lftl,\minot) & (\rtone,\minot) & 1
\\
  \mij\ & \vline & (\lftl,\mij) & (\rtone,\mij) & 1 \\
  \hline
\end{tabular}
\label{wx0table}
\end{table}

\begin{figure}[t]
\centering
\includegraphics[width=2.5in]{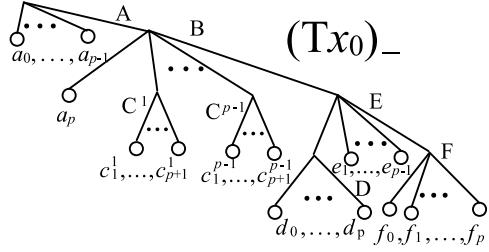}
\caption{$(Tx_0)_-$ (where $(Tx_0)_+=T_+$).}\label{wxnot}
\end{figure}

\item \label{C} Conditions on $\wedge_{C^i}$ in $(T_-,T_+)$
for $i=1,2,...,p-1$: We know from Remark
\ref{caretchange} that only $x_i$ will change the type
of $\wedge_{C^i}$ in the negative tree, from type \mi\
to type \rt\ (see Figure \ref{wxi}). First we enumerate
the conditions imposed by the specific subtype of
$\wedge_{C^i}$ in $T_-$ on the specific subtype of
$\wedge_{C^i}$ in $(Tx_i)_-$ (in Figure \ref{wxi}).
First we note that in both $T_-$ and $(Tx_i)_-$, in
$\wedge_{C^i}$, the child carets in the subtrees
$c_1^i,...,c_{p-i+1}^i$ (if they are nonempty) will be
predecessors of $\wedge_{C^i}$ and the child carets in
the subtrees $c_{p-i+2}^i,...,c_{p+1}^i$ (if they are
nonempty) will be successors of $\wedge_{C^i}$ (see
Figure \ref{type2}).  Additionally, the
root caret of the subtrees $c_1^i,...,c_{p-i+1}^i$ (if
they exist) will have caret types $\m^i,...,\m^p$
respectively, and the root carets of the subtrees
$c_{p-i+2}^i,...,c_{p+1}^i$ (if they exist) will have
caret types $\m^1,...,\m^i$ respectively (see Figure
\ref{type}).

    \begin {enumerate}
        \item If $\tau_{T_-}(\wedge_{C^i})=\minot$, then
the subtrees $c_{p-i+2}^i,...,c_{p+1}^i$ are
all empty, which implies that
$\wedge_{C^{i+1}}$ is the leftmost child
successor of $\wedge_{C^i}$.  Since
$\tau(\wedge_{C^{i+1}})=\m^{i+1}$,
$\tau_{(Tx_i)_-}(\wedge_{C^i})=\riplusone$.
        \item If $\tau_{T_-}(\wedge_{C^i})=\mij$, then the
leftmost child successor of $\wedge_{C^i}$ in
$T_-$ is the root caret of the subtree $c_j^i$,
which implies that the subtrees
$c_{p-i+2}^i,...,c_{j-1}^i$ are all empty.  So
the leftmost child successor of $\wedge_{C^i}$
in $(Tx_i)_-$ will also be the root of subtree
$c_j^i$, which is of type $\m^j$, so
$\tau_{(Tx_i)_-}(\wedge_{C^i})=\rtj$.
    \end{enumerate}

    Table \ref{wxitable} lists the change in weight (taken from
Table \ref{weighttable}) of this caret pair $\wedge_{C^i}$
when $\tau_{T_-}(\wedge_{C^i})=\minot$; When
$\tau_{T_-}(\wedge_{C^i})=\mij$, the change in caret type
of $\wedge_{C^i}$ from \mij\ to \rtj\ results in a decrease
in caret weight no matter what the type of $\wedge_{C^i}$
in $T_+$, so we conclude that if
$\tau_{T_-}(\wedge_{C^i})=\mij$, then $\wedge_{C^i}$ in
$T_+$ may be of any type.  If
$\tau_{T_-}(\wedge_{C^i})=\minot$, then we can see from Table
\ref{wxitable} that $\wedge_{C^i}$ in $T_+$ may be of type
\lftl, \rtk\ or  \mknot, or \mrs\ for $k,r,s\le i$.

\begin {table}[t]
\centering \caption{How $x_i$ (for $i=1,2,...,p-1$), when
$\tau_{T_-}(\wedge_{C^i})=\minot$, ``acts" on $w(\wedge_{C^i})$
in arbitrary dead end $w=(T_-,T_+)$, listed by possible types
of $\wedge_{C^i}\in T_+$.  }
\begin{tabular}{clccc}
  \hline
  $\tau_{T_+}(\wedge_{C^i})$ & \vline &
$\tau_{(T_-,T_+)}(\wedge_{C^i})$ &
$\tau_{((Tx_i)_-,(Tx_i)_+)}(\wedge_{C^i})$ &
$\Delta_{x_i}(\wedge_{C^i})$ \\
  \hline
  \lftl\ & \vline & (\minot,\lftl) & (\riplusone,\lftl) & -1
\\
  \rtnot\ & \vline & (\minot,\rtnot) & (\riplusone,\rtnot) & 1
\\
  \rtr\ & \vline & (\minot,\rtr) & (\riplusone,\rtr) & 1 \\
  \rtj\ & \vline & (\minot,\rtj) & (\riplusone,\rtj) &
$\displaystyle^{-1 \hbox{ for } j\le i}_{\hbox{
}\hbox{ }1 \hbox{ for } j>i}$ \\
  \mknot\ & \vline & (\minot,\mknot) & (\riplusone,\mknot) &
$\displaystyle^{-1 \hbox{ for } k\le i}_{\hbox{
}\hbox{ }1 \hbox{ for } k>i}$ \\
  $\m^l_m$ & \vline & (\minot,$\m^l_m$) & (\riplusone,$\m^l_m$)
& $\displaystyle^{-1 \hbox{ for } l\le i}_{\hbox{
}\hbox{ }1 \hbox{ for } l>i}$ \\
  \hline
\end{tabular}
\label{wxitable}
\end{table}

\begin{figure}[t]
\centering
\includegraphics[width=3in]{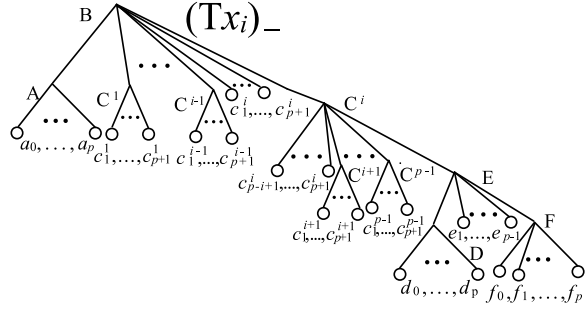}
\caption{$(Tx_i)_-$ when $i=1,...,p-1$
(where $(Tx_i)_+=T_+$).  }\label{wxi}
\end{figure}

\item \label{D} Conditions on $\wedge_D$ in $(T_-,T_+)$: We
    know from Remark \ref{caretchange} that only $x_p$ will
    change the type of $\wedge_D$ in the negative tree, from
    type \mpp\ to type \rt\ (see Figure \ref{wxp}). First we
    enumerate the conditions which determine the subtype of
    $\wedge_D$ in $T_-$ and the conditions imposed by that
    specific subtype of $\wedge_D$ in $T_-$ on the specific
    subtype of $\wedge_D$ in $(Tx_i)_-$ in Figure \ref{wxp}.  First we note that in both $T_-$ and $(Tx_p)_-$, in
$\wedge_D$, the child carets in the subtree $d_0$ (if nonempty)
will be predecessors of $\wedge_D$ and the child carets in the
subtrees $d_1,...,d_p$ (if nonempty) will be successors of
$\wedge_D$ (see Figure \ref{type2}). Additionally, the root
caret of the subtrees $d_0,d_1...,d_p$ (if they exist) will
have caret types $\m^p,\m^1,...,\m^p$ respectively (see Figure
\ref{type}).

    \begin {enumerate}
        \item If $d_j$ is a leaf for all $j\in\{1,...,p\}$,
then $\tau_{T_-}(\wedge_D)=\mpnot$, because $\wedge_D\in T_-$
will have no child successors (see Figure \ref{deadenddiag}),
and $\tau_{(Tx_p)_-}(\wedge_D)=\rtr$ or \rtnot\ since the
leftmost child successor of $\wedge_D\in (Tx_p)_-$ will be
$\wedge_E$, which will also be $\wedge_D$'s immediate successor
(see Figure \ref{wxp}).
        \item If there is a $j\in \{1,...,p\}$ such that $d_j$
is not a leaf, then $\tau_{T_-}(\wedge_D)=\mpi$, where
$i=min\{j|d_j\hbox{ is not a leaf}\}$, and
$\tau_{(Tx_p)_-}(\wedge_D)=\rti$, because  when $j<p$, the root
of the subtree $d_i$ will be the leftmost child successor of
$\wedge_D$ in both $T_-$ and $(Tx_p)_-$, and will be of type
$\m^i$ in both trees, and when $j=p$, the leftmost child
successor of $\wedge_D$ will be the root of the subtree $d_i$
(type $\m^i$) in $T_-$ and will be $\wedge_E$ (type \rt) in
$(Tx_p)_-$, and the immediate successor of $\wedge_D$ will be
in the subtree $\m^i$ in both trees (see Figures \ref{type} and
\ref{type2}).
    \end{enumerate}

    Table \ref{wxptable} lists the change in weight (taken from
Table \ref{weighttable}) of this caret pair $\wedge_D$ when
$\tau_{T_-}(\wedge_D)=\mpnot$; When
$\tau_{T_-}(\wedge_D)=\mpj$, the change in caret type of
$\wedge_D$ from \mpj\ to \rtj\ decreases caret weight, no
matter what the type of $\wedge_D$ in $T_+$, so if
$\tau_{T_-}(\wedge_D)=\mpj$, then $\wedge_D$ in $T_+$ may be of
any type.  If $\tau_{T_-}(\wedge_D)=\mpnot$, then we can see
from Table \ref{wxptable} that $\wedge_D$ in $T_+$ must be of
type \lftl, \rtj, \mknot\ or \mkl, where
$j,k,l\in\{1,2,...,p\}$. $|wx_p|<|w|$.

\begin {table}[t]
\centering \caption{How $x_p$, when
$\tau_{T_-}(\wedge_D)=\mpnot$, ``acts" on $w(\wedge_D)$ in
arbitrary dead end $w=(T_-,T_+)$, listed by possible types of
$\wedge_D\in T_+$.  Case 1 is when
$\tau_{(Tx_p)_-}(\wedge_D)=\rtr$, and case 2 is when
$\tau_{(Tx_p)_-}(\wedge_D)=\rtnot$.}
\begin{tabular}{clccccc}
  \hline
  $\tau_{T_+}(\wedge_D)$ & \vline &
$\tau_{(T_-,T_+)}(\wedge_D)$ &
\multicolumn{2}{c}{$\tau_{((Tx_p)_-,(Tx_p)_+)}(\wedge_{D})$}
& \multicolumn{2}{c}{$\Delta_{x_p}(\wedge_D)$} \\
     & \vline &   & case 1 & case 2 & case 1 & case 2 \\
  \hline
  \lftl\ & \vline & (\mpnot,\lftl) & (\rtr,\lftl) &
(\rtnot,\lftl) & -1 & -1 \\
  \rtnot\ & \vline & (\mpnot,\rtnot) & (\rtr,\rtnot) &
(\rtnot,\rtnot) & 1 & 1\\
  \rtr\ & \vline & (\mpnot,\rtr) & (\rtr,\rtr) &
(\rtnot,\rtr) & 1 & 1 \\
  \rtj\ & \vline & (\mpnot,\rtj) & (\rtr,\rtj) &
(\rtnot,\rtj) & -1 & -1 \\
  \minot\ & \vline & (\mpnot,\minot) & (\rtr,\minot) &
(\rtnot,\minot) & -1 & -1 \\
  \mij\ & \vline & (\mpnot,\mij) & (\rtr,\mij) &
(\rtnot,\mij) & -1 & -1 \\
  \hline
\end{tabular}
\label{wxptable}
\end{table}

\begin{figure}[t]
\centering
\includegraphics[width=2.75in]{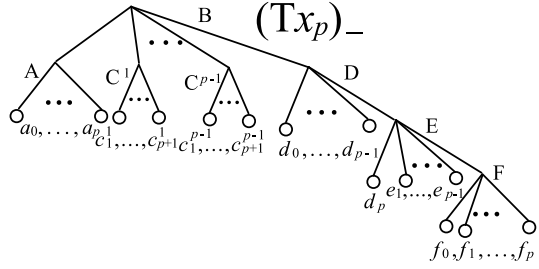}
\caption{$(Tx_p)_-$ (where $(Tx_p)_+=T_+$).  }\label{wxp}
\end{figure}

\item \label{E} Conditions on $\wedge_E$ in $(T_-,T_+)$: We
    know from Remark \ref{caretchange} that $x_i^{-1}$ for
    $i=0,1,2,...,p$ will change the type of $\wedge_E$ in the
    negative tree, from type \rt\ to type \lftl\ when $i=0$
    (see Figure \ref{wxnotinv}) and \mi\ when $i>0$ (see
    Figures \ref{wxiinv} and \ref{wxpinv}). First we enumerate
    the conditions that determine the subtype of
    $\wedge_E$ in $T_-$ (which is $\rt$) and in $(Tx_i)_-$,
    (which is \lftl\ when $i=1$ and \mi\ when $i>0$)
    by considering Figures \ref{negtree}, \ref{wxnotinv}, \ref{wxiinv} and \ref{wxpinv}.  To understand this set of
    conditions, see Figures \ref{type} and \ref{type2}.  Here we also define $e_p=f_0$.

\begin {enumerate}
\item If $e_k$ is a non-empty subtree in $T_-$ for some $k\in
    \{1,...,p\}$, then:
    \begin{enumerate}
    \item The type of $\wedge_E$ in $T_-$ is \rtj\ (where
        $j=min\{k|e_k\hbox{ is nonempty}\}$), because when
        $j<p$, the root of $e_j$ (which is type $\m^j$)
        will be the leftmost child successor of $\wedge_E$,
        and when $j=p$, $\wedge_F$ (which is type \rt) will
        be the leftmost child successor of $\wedge_E$ and
        the immediate successor of $\wedge_E$ will be in
        $e_j$ (and thus not type \rt).
        \item They type of $\wedge_E$ in $(Tx_i^{-1})_-$ is
            \lftl\ for $i=0$ and \mij\ for $i>0$, because
            the leftmost child successor of $\wedge_E$ in
            $(Tx_i^{-1})_-$ is the root of the subtree
            $e_j$, which is of type $\m^j$ (see Figures
            \ref{type} and \ref{type2}).
    \end{enumerate}
\item If $e_k$ is a leaf in $T_-$ for all $k\in \{1,...,p\}$,
    then $\wedge_F$ (which is type \rt) will be the
        immediate successor of $\wedge_E$ in both $T_-$ and $(Tx_i^{-1})_-$:
    \begin{enumerate}
    \item The type of $\wedge_E$ in $T_-$ is \rtnot\ when
        $\wedge_F$ in $T_-$ is type \rtnot\ and \rtr\
        otherwise.  If $\wedge_F$ in $T_-$ is type \rtnot,
        then all of the successors of $\wedge_F$ are type
        \rt, and thus all successors of $\wedge_E$ must
        also be type \rt.  If $\wedge_F$ in $T_-$ is not of
        type \rtnot, then there exists at least one
        successor of $\wedge_F$, and of $\wedge_E$ by
        extension, which is not of type \rt.
        \item The type of $\wedge_E$ in $(Tx_i^{-1})_-$ is
            \lftl\ for $i=0$ and \minot\ for $i>0$,
            because $\wedge_E$ will have no nonempty child
            successor in $(Tx_i^{-1})_-$.
        \end{enumerate}
\end{enumerate}

    \par Table \ref{wx0invtable} lists the change in weight
(taken from Table \ref{weighttable}) of $\wedge_E$ when $i=0$,
and Table \ref{wxiinvtable} lists the change in weight of
$\wedge_E$ when $i>0$.  So now we proceed to outline the
possible caret types of $\wedge_E$ in $(T_-,T_+)$ which result
in reduced length after multiplication by $x_i^{-1}$ for
$i=0,...,p$.

\begin {table}[t]
\addtolength{\tabcolsep}{-2pt} \centering \caption{How
$x_0^{-1}$ ``acts" on $w(\wedge_E)$ in arbitrary dead end
$w=(T_-,T_+)$, listed by possible types of $\wedge_E\in T_+$.
Case 1 is when $\tau_{T_-}(\wedge_E)=\rtnot$, case 2 is when
$\tau_{T_-}(\wedge_E)=\rtr$, and case 3 is when
$\tau_{T_-}(\wedge_E)=\rtj$.}
\begin{tabular}{clccccccc}
  \hline
  $\tau_{T_+}(\wedge_E)$ & \vline &
\multicolumn{3}{c}{$\tau_{(T_-,T_+)}(\wedge_E)$} &
$\tau(\wedge_E)$ &
\multicolumn{3}{c}{$\Delta_{x_0^{-1}}(\wedge_E)$} \\
     & \vline & case 1 & case 2 & case 3 & in $wx_0^{-1}$ &
case 1 & case 2 & case 3 \\
  \hline
  \lftl\ & \vline & (\rtnot,\lftl) & (\rtr,\lftl) &
(\rtj,\lftl) & (\lftl,\lftl) & 1 & 1 & 1 \\
  \rtnot\ & \vline & (\rtnot,\rtnot) & (\rtr,\rtnot) &
(\rtj,\rtnot) & (\lftl,\rtnot) & 1 & -1 & -1 \\
  \rtr\ & \vline & (\rtnot,\rtr) & (\rtr,\rtr) & (\rtj,\rtr) &
(\lftl,\rtr) & -1 & -1 & -1\\
  \rtk\ & \vline & (\rtnot,\rtk) & (\rtr,\rtk) & (\rtj,\rtk) &
(\lftl,\rtk) & -1 & -1 & -1 \\
  \minot\ & \vline & (\rtnot,\minot) & (\rtr,\minot) &
(\rtj,\minot) & (\lftl,\minot) & 1 & 1 &
$\displaystyle^{\hbox{ }\hbox{ }1 \hbox{ for } i<j}_{-1
\hbox{ for } i\ge j}$ \\
  \mlk\ & \vline & (\rtnot,\mlk) & (\rtr,\mlk) & (\rtj,\mlk) &
(\lftl,\mlk) & -1 & -1 & -1 \\
  \hline
\end{tabular}
\label{wx0invtable}
\end{table}

\begin {table}[t]
\addtolength{\tabcolsep}{-2pt} \centering \caption{How
$x_i^{-1}$ (for $i=1,2,...,p$) ``acts" on $w(\wedge_E)$ in
arbitrary dead end $w=(T_-,T_+)$, listed by possible types of
$\wedge_E\in T_+$.  Case 1 is when
$\tau_{T_-}(\wedge_E)=\rtnot$, case 2 is when
$\tau_{T_-}(\wedge_E)=\rtr$, and case 3 is when
$\tau_{T_-}(\wedge_E)=\rtj$ (where $j>i$).}
\begin{tabular}{clccccccc}
  \hline
  $\tau_{T_+}(\wedge_E)$ & \vline &
\multicolumn{3}{c}{$\tau_{(T_-,T_+)}(\wedge_E)$} &
$\tau(\wedge_E)$ &
\multicolumn{3}{c}{$\Delta_{x_i^{-1}}(\wedge_E)$} \\
     & \vline & case 1 & case 2 & case 3 & in $wx_i^{-1}$ &
case 1 & case 2 & case 3 \\
  \hline
  \lftl\ & \vline & (\rtnot,\lftl) & (\rtr,\lftl) &
(\rtj,\lftl) & (\minot,\lftl) & 1 & 1 & 1 \\
  \rtnot\ & \vline & (\rtnot,\rtnot) & (\rtr,\rtnot) &
(\rtj,\rtnot) & (\minot,\rtnot) & 1 & -1 & -1 \\
  \rtr\ & \vline & (\rtnot,\rtr) & (\rtr,\rtr) & (\rtj,\rtr) &
(\minot,\rtr) & -1 & -1 & -1\\
  \rtk\ & \vline & (\rtnot,\rtk) & (\rtr,\rtk) & (\rtj,\rtk) &
(\minot,\rtk) & \multicolumn{3}{c}{$1 \hbox{ for } k\le i, -1
\hbox{ for } k>i$} \\
  $\m^l_\emptyset$ & \vline & (\rtnot,$\m^l_\emptyset$) &
(\rtr,$\m^l_\emptyset$) & (\rtj,$\m^l_\emptyset$) &
(\minot,$\m^l_\emptyset$) & 1 & 1 & $\displaystyle^{\hbox{
}\hbox{ }1 \hbox{ for } l<j}_{-1 \hbox{ for } l\ge j}$ \\
  $\m^m_n$ & \vline & (\rtnot,$\m^m_n$) & (\rtr,$\m^m_n$) &
(\rtj,$\m^m_n$) & (\minot,$\m^m_n$) & \multicolumn{3}{c}{$1
\hbox{ for } m\le i, -1 \hbox{ for } m>i$} \\
  \hline
\end{tabular}
\label{wxiinvtable}
\end{table}

\begin{figure}[t]
\centering
\includegraphics[width=2.75in]{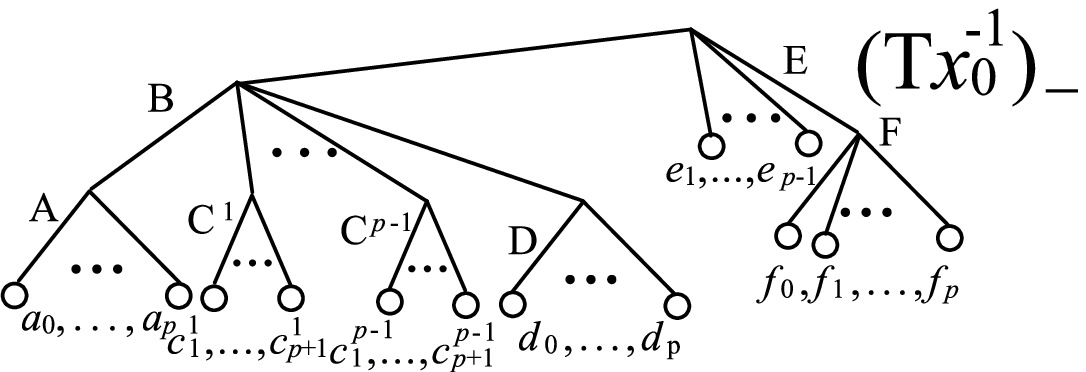}
\caption{$(Tx_0^{-1})_-$ (where $(Tx_0^{-1})_+=T_+$).}\label{wxnotinv}
\end{figure}

\begin{figure}[t]
\centering
\includegraphics[width=3in]{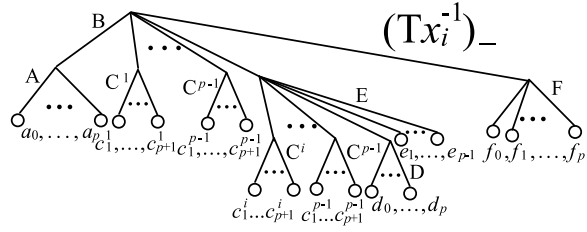}
\caption{$(Tx_i^{-1})_-$ (where $(Tx_i^{-1})_+=T_+$).}\label{wxiinv}
\end{figure}

\begin{figure}[t]
\centering
\includegraphics[width=2.25in]{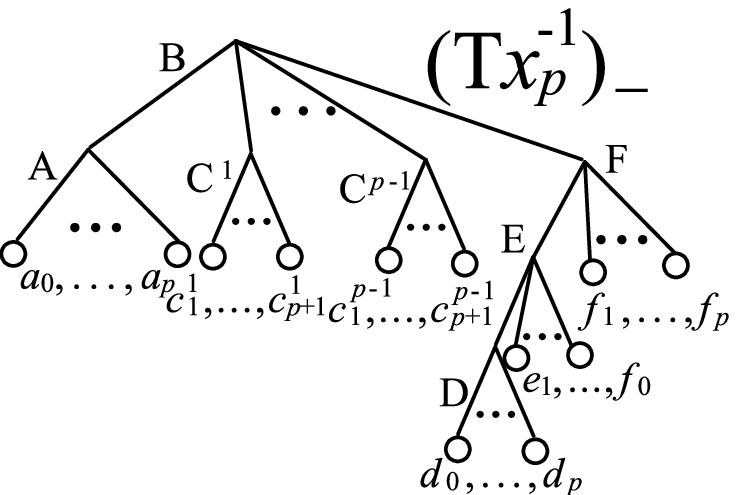}
\caption{$(Tx_p^{-1})_-$ (where $(Tx_p^{-1})_+=T_+$).}\label{wxpinv}
\end{figure}

\par From Tables \ref{wx0invtable} and \ref{wxiinvtable}, we
have the following sets of conditions.
\begin {enumerate}
\Alph{enumii} \item \label{iequals0} $i=0$: The possible caret
pairings for $\wedge_E$ in $(T_-,T_+)$, determined because the
weight of $\wedge_E$ decreases after multiplication by
$x_0^{-1}$ (see Table \ref{wx0invtable}) are:
\begin{enumerate}
\item $(\rt,\rt)$ excluding $(\rtnot,\rtnot)$ \item
$(\rt,\m^t_u)$ \item $(\rtj,\mlnot)$ such that $l\ge
j$
\end{enumerate}

\item \label{igreaterthan0} $i>0$: We define the variable
$\rt'\in \{\rtnot,\rtr,\rtj|j>i\}$.  The possible caret
type pairs for $\wedge_E$ in $(T_-,T_+)$, determined
because the weight of $\wedge_E$ decreases after
multiplication by $x_i^{-1}$ where $i\in\{1,2,3,...,p\}$
(see Table \ref{wxiinvtable}) are:
\begin{enumerate}
\item (\rtr,\rtnot) \item (\rtj,\rtnot) where $j>i$ \item
$(\rt',\rtr)$ \item $(\rt',\rtk)$ where $k>i$ \item
(\rtj, \mlnot) where $j>i$ and $l\ge j$ \item
$(\rt',\mrs)$ where $s>i$ (and if $\rt'=\rtj$, then
$r\ge j$)
\end{enumerate}
\end{enumerate}
\par We note that multiplying by {\it each} $x_i^{-1}$ for
$i=0,1,2,...,p$ imposes its own set of conditions on the type
pair of $\wedge_E$. In order for $w$ to be a dead end, the
caret $\wedge_E$ in $w=(T_-,T_+)$ must satisfy {\it all} $p+1$
sets of conditions, because its length must be reduced whenever
we multiply by $x_i^{-1}$ for {\it any} $i\in\{0,...,p\}$. We
note that
$\cap\displaystyle^p_{i=0}\left\{(\rtj,*)|j>i\right\}=\emptyset$
and
$\cap\displaystyle^p_{i=0}\left\{(\mrs,*)|s>i\right\}=\emptyset$
for any caret type $*$, so taking the intersection of the set
of possible caret type pairs for all $i\in\{0,...,p\}$ given in
\ref{iequals0} and \ref{igreaterthan0} yields:
$$(\rtnot,\rtr), (\rtr,\rtnot), (\rtr,\rtr)$$

\par These are the only type pairs for $\wedge_E$ which will
result in $|wx_i^{-1}|<|w|$ for {\it all} $i\in\{0,...,p\}$, and since $\wedge_E$
is of type \rtnot\ or \rtr\ in both $T_-$ and $T_+$, each
$e_1,...,e_p$ must be a leaf in both $T_-$ and $T_+$.

\item Conditions on $\wedge_F$ in $(T_-,T_+)$: We know from
    Remark \ref{caretchange} that there is no $g\in X\cup
    X^{-1}$ which will change the type of $\wedge_F$ in the
    negative tree, so we have no conditions on the type of this
    caret unless they are imposed by the required types of
    other carets within the tree.  By definition $\wedge_F$ is
    of type \rt\ in $T_-$.  Since $e_1,...,e_{p-1},f_0$ must
    all be leaves in $T_-$ and $T_+$ (see \ref{E}), $\wedge_F$ is the
    immediate successor of $\wedge_E$, so $\wedge_F$
    must be type \rt\ in $T_+$.
\end{enumerate}

\par We summarize the possible caret pairings outlined above
for each of the labeled
carets in $(T_-,T_+)$ in Table \ref{summary}.  These are precisely the conditions met by Figure \ref{deadend}.

\begin {table}[t]
\centering \caption{Possible caret pairings for labeled carets
in a dead end $w=(T_-,T_+)$.  Here * can be any caret type.  }
\begin{tabular}{cclccc}
  \hline
  $\wedge_A$ & $\wedge_B$ & $\wedge_{C^i}$, $i-1,...,p-1$ &
$\wedge_D$ & $\wedge_E$ & $\wedge_F$ \\
  \hline
  (\lft,\lft) & (\lftl,\lftl) & (\mi,\lftl) & (\mpp,*), &
$(\rtnot,\rtr)$ & (\rt,\rt) \\
  (\lftl,\m) &   & (\minot,\rtk) for $k\le i$ & except &
$(\rtr,\rtnot)$ &   \\
    &   & (\minot,\mlnot) for $l\le i$ & (\mpnot,\rtr) &
$(\rtr,\rtr)$ &    \\
    &   & (\minot,\mrs) for $r,s\le i$ & or &   &   \\
    &   & (\mij,*) & (\mpnot, \rtnot) &   &   \\
  \hline
\end{tabular}
\label{summary}
\end{table}
\end{proof}

\subsection{Depth of dead ends}
\begin{depthdeadends}
All dead ends in $F(p+1)$ have depth 2 with respect to $X$. Or, there
are no k--pockets in $F(p+1)$ for $k\ne2$.
\end{depthdeadends}

\begin{proof}
\par We show that for arbitrary dead end
$w$, $|wx_0^{-1}x_ix_j|$ for any $i,j\in\{1,2,...,p\}$
will have length greater than $|w|$.  The word
$wx_0^{-1}x_1^2$ which Cleary and Taback use in
\cite{combpropF} to prove this theorem for $p=1$ is a
subcase of this construction.

\par Suppose
$|w|=q$; we have seen
that $|wg^{\pm 1}|=q-1$ for $g\in \gen$.  So
$|wg_1^{\pm 1}g_2^{\pm 1}|\le q$ for $g_1,g_2\in \gen$, which
shows that $w$ cannot have depth 1, and $|wg_1^{\pm 1}g_2^{\pm
1}g_3^{\pm 1}|\le q+1$ for $g_1,g_2,g_3\in \gen$.  So, to show
that a dead end $w$ in $F(p+1)$ has depth 2, we need only find
$g_1,g_2,g_3\in \gen$ such that
$|wg_1^{\epsilon_1}g_2^{\epsilon_2}g_3^{\epsilon_3}|\ge q+1$
where $\epsilon_1,\epsilon_2,\epsilon_3\in \{-1,1\}$.

\par If we consider the tree-pair diagram for $w$ given in
Figure
\ref{deadend}, we can see that $wx_0^{-1}$
will have the tree-pair diagram given in Figure
\ref{wxnotinv}.  $|wx_0^{-1}|=q-1$, and to
multiply
$wx_0^{-1}$ by $x_i$ for $i=1,2,...,p$, we must add a
caret to the tree-pair diagram for $wx_0^{-1}$ on the leaf with
index number $e_i$ (note: for $i=p$, we use the convention
$e_p=f_0$); we call this new caret $E^i$. So the tree-pair
diagram for $wx_0^{-1}x_i$ will have the form given in Figure
\ref{wxnotinvxi}.  Since we had to add a caret to the tree-pair
diagram
for $wx_0^{-1}$ to get $wx_0^{-1}x_i$, by Theorem
\ref{addcaret}, $|wx_0^{-1}x_i|\ge q$. To
multiply $wx_0^{-1}x_i$ by $x_j$ where $j=1,2,...,p$, we
need to add a caret to the tree-pair diagram for $wx_0^{-1}x_i$
on the leaf with index number $e_j$, and then by Theorem
\ref{addcaret}, $|wx_0^{-1}x_ix_j|>q$. Therefore all dead ends
have depth 2 in $F(p+1)$ under $X$.

\begin{figure}[t]
\centering
\includegraphics[width=4in]{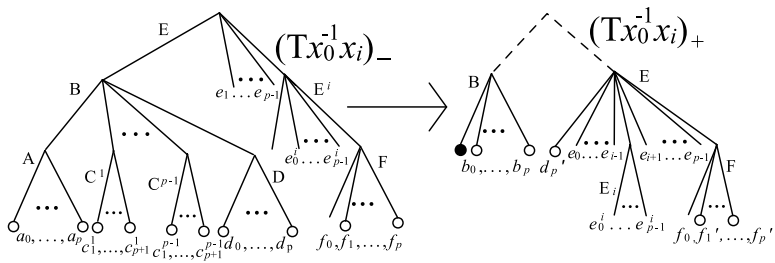}
\caption{Tree-pair diagram representative of $wx_0^{-1}x_i$,
for $i=1,2,...,p$ and $w$ a dead end in $F(p+1)$.
}\label{wxnotinvxi}
\end{figure}

\end{proof}

\bibliography{seesawbib}

\providecommand{\bysame}{\leavevmode\hbox to3em{\hrulefill}\thinspace}
\providecommand{\MR}{\relax\ifhmode\unskip\space\fi MR }
\providecommand{\MRhref}[2]{%
  \href{http://www.ams.org/mathscinet-getitem?mr=#1}{#2}
}
\providecommand{\href}[2]{#2}
\begin{thebibliography}{10}

\bibitem{F2notMAC}
James Belk and Kai-Uwe Bux, \emph{{Thompson's group $F$ is maximally
  nonconvex}}, Contemp. Math. \textbf{372} (2005), 131--146.

\bibitem{bogo}
O.V. Bogopol'ski\v{\i}, \emph{{Infinite commensurable hyperbolic groups are
  bi-Lipschitz equivalent}}, Algebra and Logic \textbf{36(3)} (1997), 155--163.

\bibitem{Brown}
K.S. Brown, \emph{Finiteness properties of groups}, J. Pure App. Algebra
  \textbf{44} (1987), 45--75.

\bibitem{Brown2}
K.S. Brown and Geoghegan. R., \emph{{An infinite-dimensional torsion-free $FP1$
  group}}, Invent. Math. \textbf{77} (1984), 367--381.

\bibitem{intronotes}
J.W. Cannon, W.J. Floyd, and W.R. Parry, \emph{{Introductory notes on Richard
  Thompson's groups}}, Enseign. Math. \textbf{42} (1996), 215--256.

\bibitem{geolang}
Sean Cleary, Murray Elder, and Jennifer Taback, \emph{Cone types and geodesic
  languages for lamplighter groups and thompson's group $f$}, J. Algebra
  \textbf{303(2)} (2006), 476--500.

\bibitem{lamplight}
Sean Cleary and Tim~R. Riley, \emph{{A finitely presented group with unbounded
  dead-end depth}}, Proc. Amer. Math. Soc. \textbf{134(2)} (2006), 343--349.

\bibitem{FnotAC}
Sean Cleary and Jennifer Taback, \emph{{Thompson's group $F$ is not almost
  convex}}, J. Algebra \textbf{270} (2003), 133--149.

\bibitem{combpropF}
\bysame, \emph{{Combinatorial properties of Thompson's group $F$}}, Trans.
  Amer. Math. Soc. \textbf{356} (2004), 2825--2849.

\bibitem{seesaw}
\bysame, \emph{{Seesaw words in Thompson's group $F$}}, Contemp. Math.
  \textbf{372} (2005), 147--159.

\bibitem{length}
S.~Blake Fordham, \emph{{Minimal length elements of $F(p)$}}, Preprint.

\bibitem{fordhamthesis}
\bysame, \emph{Minimal length elements of {T}hompson's group ${F}$}, Ph.D.
  thesis, Brigham Young University, 1995.

\bibitem{fordhamgd}
\bysame, \emph{Minimal length elements of {T}hompson's group {$F$}}, Geom.
  Dedicata \textbf{99} (2003), 179--220. \MR{MR1998934 (2004g:20045)}

\bibitem{Higman}
G.~Higman, \emph{Finitely presented infinite simple groups}, Notes on Pure
  Math. \textbf{8} (1974).

\bibitem{F}
R.~J. Thompson and R.~McKenzie, \emph{An elementary construction of unsolvable
  word problems in group theory}, Word problems, Conference at University of
  California, Irvine, North Holland, 1969 (1973).

\bibitem{notMAC}
Claire Wladis, \emph{{Thompson's group $F(n)$ is not minimally almost convex}},
  New York J. Math. \textbf{13} (2007), 437--481.

\end{thebibliography}
\bibliographystyle{amsplain}

\end{document}